\documentclass[11pt,twoside,a4paper]{article}
\usepackage{amsmath}
\usepackage{amssymb}
\usepackage{amsthm}
\usepackage{enumerate}
\usepackage{color}
\usepackage{bm}

\newcommand{\field}[1]{\mathbb{#1}}

\newcommand{\R}{\field{R}}

\renewcommand{\H}{\field{H}}
\newcommand{\N}{\field{N}}

\newcommand{\la}{\label}
\newcommand{\pref}[1]{(\ref{#1})}
\newcommand{\ba}{\begin{array}}
\newcommand{\ea}{\end{array}}

\newcommand{\vep}{\varepsilon}

\newcommand{\be}{\begin{equation}}
\newcommand{\ee}{\end{equation}}

\newcommand{\bea}{\begin{eqnarray}}
\newcommand{\eea}{\end{eqnarray}}

\newcommand{\nn}{\nonumber}

\newcommand{\dist}{{\mathrm{dist}}}

\newcommand{\supp}{{\mathrm{supp}}}
\newcommand{\loc}{\mathrm{loc}}

\newcommand{\intmean}{{\int\hspace{-11.6pt}-}}
\newcommand{\sintmean}{{\int\hspace{-9pt}-}}

\newcounter{const}

\newtheorem{thm}{Theorem}[section]
\newtheorem{defi}[thm]{Definition}

\newtheorem{prop}[thm]{Proposition}
\newtheorem{cor}[thm]{Corollary}
\newtheorem{rem}[thm]{Remark}

\renewcommand{\H}{{\cal H}}
\newcommand{\G}{{\cal G}}
\newcommand{\F}{{\cal F}}

\renewcommand{\epsilon}{\varepsilon}

 \def\ocirc#1{\ifmmode\setbox0=\hbox{$#1$}\dimen0=\ht0
    \advance\dimen0 by1pt\rlap{\hbox to\wd0{\hss\raise\dimen0
    \hbox{\hskip.2em$\scriptscriptstyle\circ$}\hss}}#1\else
    {\accent"17 #1}\fi}

\def\cprime{$'$}
\allowdisplaybreaks[1]
\def\cb{\color{black}}
\def\cred{\color{red}}
\def\cbl {\color{blue}}

\begin{document}

\title{
Regularity for minimizers for functionals of double phase with variable exponents
}
\author{
Maria Alessandra Ragusa
\thanks{Dipartimento di Matematica e Informatica,
Viale Andrea Doria, 6-95125 Catania, Italy,
%\cm
 RUDN University”, 6 Miklukho - Maklay St, Moscow, 117198, Russia
%\cb
e-mail:maragusa@dmi.unict.it}
~\&~ Atsushi Tachikawa
\thanks{Department of Mathematics, Faculty of Science and Technology,
Tokyo University of Science, Noda, Chiba, 278-8510, Japan,
e-mail:tachikawa$\_$atsushi@ma.noda.tus.ac.jp}
\thanks{
 The first author is partially supported by PRIN 2017 and the Ministry of Education and Science of the Russian Federation (5-100 program of the Russian Ministry of Education).
The second author is partially supported by Japan Society for the Promotion of Science
 KAKENHI Grant Number 17K05337.}
}
\date{\empty}
\maketitle
\begin{abstract}
The functionals of double phase type
\[
	 \H (u):= \int \left(|Du|^{p} + a(x)|Du|^{q} \right) dx, ~~
	 ~~~(q>p>1,~~a(x)\geq 0)
\]
are introduced in the
epoch-making paper by Colombo-Mingione \cite{colmin15-1}
for constants $p$ and $q$, and investigated by them and Baroni.
%(e.g. \cite{colmin15-2,barcolmin16, barcolmin16(Petersburg),colmin16, barcolmin18})
They obtained sharp regularity results for minimizers of such functionals.
In this paper we treat the case that the exponents are functions of $x$
%\[
%	\F(u):=\int \left(|Du|^{p(x)} + a(x)|Du|^{q(x)} \right) dx,
%\]
 and partly generalize their regularity results.

%We show a regularity theorem for minimizers of functionals of double phase type
%with variable exponents:
%\[
%	\F(u):=\int \left(Du|^{p(x)} + a(x)|Du|^{q(x)} \right) dx,
%\]
%where $q(\cdot) > p(\cdot) >1$ and $a(\cdot)\geq 0$.
%However we assume the continuity of $p(\cdot)$ and $q(\cdot)$,
%the above functional $\F$ changes their growth order at points where
%$a(x)=0$.
%Such type of functionals have been introduced in the
%epoch-making paper by Colombo-Mingione \cite{colmin15}
%for constants $p$ and $q$, and investigated by Baroni, Colombo and Mingione.
\end{abstract}

\section{Introduction and main theorem}

The main goal of this  paper is to provide  a regularity theorem  for minimizers of a class of integral functionals of the  calculus of variations
called {\it of double phase type} with variable exponents defined for
$u \in W^{1,1}(\Omega;\R^N)~~( \Omega\in \R^n,~n,N\geq 2)$ as%and%  a
\[
	\F(u,\Omega) := \int_\Omega \left( |Du|^{p(x)} + a(x)|Du|^{q(x)} \right) dx, ~~~~
q(x)\geq p(x)>1,~ a(x)\geq 0,
\]
where $p(x), q(x)$ and $a(x)$ are assumed to be H\"older continuous.
 They do not only have strongly  non-uniform ellipticity but also
discontinuity of growth order at points where $a(x)=0$.
%The main model case in question here
The above functional is provided by the following type of functionals with variable exponent growth
\[
	u \mapsto \int g(x,Du) dx,   ~~~ \lambda |z|^{p(x)} \leq g(x,z) \leq \Lambda (1+|z|)^{p(x)},
	~~~\Lambda\geq \lambda >0,
\]
which are called  {\it  of $p(x)$-growth. }
These $p(x)$-growth functionals have been introduced by Zhikov \cite{zhi86(rus)} (in this article $\alpha(x)$ is used as variable exponents)
in the setting of Homogenization theory.
%For minimizers he showed higher integrability,  and on the other hand he gave an example with discontinuous
%\cm
%da "For minimizers..." cambierei con:
He showed higher integrability for minimizers and, on the other hand, he gave an example of discontinuous
%\cb
exponent $p(x)$ for which
the Lavrentiev phenomenon occurs (\cite{zhi95,zhi97}).

Such functionals provide a useful prototype for describing the behaviour  of strongly
inhomogeneous materials whose strengthening properties, connected to the exponent dominating the
growth of the gradient variable, significantly change with the point.
In \cite{zhi95}, Zhikov pointed out the relationship between $p(x)$-growth functionals and some
physical problems including thermistor.
As another application, the theory of electrorheological materials and fluids is known.
About these objects see, for example, \cite{ruz00(LN), acemin02,rajruz01,bogduzhabsch12}.

These kind of functionals have been the object of intensive investigation over the last years,
starting with the inspiring papers by Marcellini  \cite{mar89, mar91, mar96(SNS)}, where he introduced
so-called {\it $(p,q)$-} or {\it nonstandard} growth functionals:
\[
	u \mapsto \int f(x,u,Du) dx,~~
	\lambda |z|^p \leq f(x,u,z) \leq \Lambda (1+|z|)^q,
	~~~q \geq p \geq 1, ~~\Lambda\geq \lambda >0.
\]
About general  $(p,q)$-growth functionals, see for example \cite{BIFUCHS, BREIT, CHOE, espleomin04, sch08(advcal),
sch08(calvar),sch09,URUR, zhi95, zhi97}
 and the survey \cite{min06}.

For the continuous variable exponent case, nowadays many results on the regularity for minimizer are
known, see \cite{cosmin99,acemin01,acemin01(arc),ele04}.
Further results in this direction can be, for instance, found in
\cite{CRR}, \cite{EMM1}, \cite{EMM2}, \cite{EMM3},   \cite{ min07}, \cite{PRR}, \cite{PR1}, \cite{PR2}, \cite{PR3}, \cite{PR4}, \cite{PR5},
\cite{ragtactak,ragtac13,tac14,usu15,ragtac16,tacusu16} for partial regularity results for
$p(x)$-energy type functionals:
\[
	u \mapsto \int \left( A^{\alpha\beta}_{ij}(x,u)
	D_\alpha u^i(x) D_\beta u^j (x) \right)^{p(x)} dx, ~~~
	A^{\alpha\beta}_{ij}(x,u)z^i_\alpha z^j_\beta \geq \lambda |z|^2
\]
%{\color{magenta}{ In the context of Harnack inequality using De Giorgi's method in [HHT] is proved local H\"older continuity of minimizers but, obviously, we remark that this mentioned paper is concerning scalar value case and that the results contained in it are not valid in the vector valued case
%}}

In 2015 a new class of functional so-called \it functionals of double phase \rm
are introduced by Colombo-Mingione \cite{colmin15-1}.
In the primary model they have in mind are
\[
	u \mapsto \H(u;\Omega):= \int H(x,Du) dx, ~~H(x,z):=|z|^p+a(x)|z|^q,
\]
where $p$ and $q$ are constants with $q\geq p>1$ and $a(\cdot)$ is a
H\"older continuous non-negative function.
By Colombo-Mingione \cite{colmin15-1,colmin15-2,colmin16} and
Baroni-Colombo-Mingione \cite{barcolmin16(Petersburg),barcolmin16,barcolmin18}
many sharp results are given about the regularity of local minimizers of the functional defined as
\be\la{def-G}
	u \mapsto \G(u;\Omega):= \int_\Omega G(x,u,Du) dx,
\ee
where $G(x,u,z) :\Omega \times \R \times \R^n \to R$ is a Carath\'eodory function
satisfying the following growth condition for some constants $\Lambda \geq \lambda >0$
besides several natural assumptions:
\[
	\lambda H(x,z) \leq G(x,u,z) \leq \Lambda H(x,z).
\]

For the scalar valued case,
in \cite{barcolmin18}  regularity results are given comprehensively.
Under the conditions
\be\la{cond-a1}
	a(\cdot) \in C^{0,\alpha}(\Omega), ~~\alpha \in (0,1]~~~\mathrm{and}~~~
	\frac{q}{p}\leq 1+ \frac{\alpha}{n},
\ee
or
\be\la{cond-b}
	u\in L^\infty(\Omega), ~~a(\cdot)\in C^{0,\alpha}(\Omega), ~~\alpha \in (0,1]~~
	~\mathrm{and}~~~\frac{q}{p}\leq 1+ \frac{\alpha}{p},
\ee
they showed that a local minimizer of $\G$ defined as \pref{def-G} is in the class $C^{1,\beta}$
%-regularity
 for some $\beta \in (0,1)$.

% \cbl\fbox{AT}
%\cm
 For the scaler valued case, see also \cite{HHT}.
 They proved Harnack's inequality and the H\"olde continuity for quasiminimizer of the functional fo type
 \[
 	\int \varphi(x, |Du|) dx,
\]
where $\varphi$ is the so-called \it $\Phi$-function. \rm
We mention that Harnack's inequality is not valid in the vector valued cases
which we are considering in the present paper.
% \cb

On the other hand, for vector valued case, in \cite{colmin15-1},
under the condition
\be\la{cond-c}
	a(\cdot) \in C^{0,\alpha}(\Omega), ~~\alpha \in (0,1]~~~\mathrm{and}~~~
	\frac{q}{p}< 1+ \frac{\alpha}{n},
\ee
$C^{1,\beta}$-regularity, for some $\beta \in (0,1)$, of local minimizers is given.

Zhikov has given in \cite{zhi95,zhi97} examples of functionals
with discontinuous growth order for which Lavrentiev phenomenon occurs.
So, in general settings,
we can not expect regularity of minimizers for such functionals which change
their growth order discontinuously. So, conditions
\pref{cond-a1}, \pref{cond-b} and \pref{cond-c}, which guarantee the regularity of minimizers,
are very significant.

~~

In this paper we deal with a typical type of functionals of double phase with variable exponents
and show a regularity result for minimizers.

In our opinion these results present new and interesting features from the point of view
of regularity theory.

Let $\Omega \subset \R^n$ be a bounded domain, $p(x), q(x)$ and $a(x)$
functions on $\Omega$ satisfying
\be\la{cond-p,q-1}
	p, q \in C^{0,\sigma}(\Omega), ~~~ q(x) \geq p(x) \geq p_0 >1,
~~~\mathrm{for all }~~~ x \in \Omega
\ee
where $p_0$ is a fixed constant strictly larger than one
and
\be\la{cond-a}
	a\in C^{0,\alpha}(\Omega), ~~~ a(x) \geq 0,
\ee
for $\alpha, ~\sigma \in (0,1]$.
Moreover, we assume that $p(x)$ and $q(x)$ satisfy
\be\la{cond-p,q-2} \sup_{x \in \Omega}~	
 \frac{q(x)}{p(x)} < 1+\frac{\beta}{n},~~~~~\beta = \min\{\alpha, \sigma\},
\ee
at every $x\in \Omega$  (compare these conditions with \pref{cond-a1}).
Let $F : \Omega \times \R^{nN} \to [0,\infty)$ be a function defined by
\be\la{def-F}
	F(x,z):= |z|^{p(x)} + a(x) |z|^{q(x)}.
\ee
We consider the functional with double phase and variable exponents defined
for $u:\Omega \to \R^N$ and $D\Subset \Omega$ as
\be\la{def-ful_F}
	\F(u,D) = \int_D F(x,Du) dx.
\ee
%The aim of this paper is to study the regularity of minimizers for such functionals.

For a bounded open set $\Omega\subset \R^n$ and a function
$p : \Omega \to [1, +\infty)$, we define $L^{p(x)}(\Omega;\R^N)$
and $W^{1,p(x)}(\Omega;\R^N)$ as follows:
\begin{align*}
	L^{p(x)} (\Omega;\R^N)
	& := \{  u \in L^1(\Omega;\R^N)~;~ \int_\Omega |u|^{p(x)} dx < +\infty\}. \\
	W^{1,p(x)}(\Omega; \R^N) &:=
	\{ u \in L^{p(x)}\cap W^{1,1}(\Omega;\R^N)~;~ Du \in L^{p(x)}(\Omega;\R^{nN})\}.
\end{align*}
In what follows we omit the target space $\R^N$.
We also define $L^{p(x)}_{\mathrm{loc}}(\Omega)$ and
$W^{1,p(x)}_{\mathrm{loc}}(\Omega)$ similarly.
As mentioned in \cite{cosmuc02}, if $p(x)$ is uniformly continuous and
$\partial \Omega$ satisfies uniform cone property, then
\[
    W^{1, p(x)}(\Omega)=\{u \in W^{1,1}
    (\Omega)~ ;  Du \in L^{p(x)}(\Omega) \}.
\]
%In any case, if $p(x)$ is continuous in $\Omega$, we have
%\[	
%	W_{\mathrm{loc}}^{1, p(x)}(\Omega)
%   	 = \{u \in W_{\mathrm{loc}}^{1,1}(\Omega) \,;  |Du|^{p(x)}\in
%	 L^1_{\mathrm{loc}}(\Omega) \}.
%\]
%
%We also define
%\[
%    W^{1, p(x)}_0(\Omega):=\{ u\in W^{1,1}_0(\Omega)~;~
%    \int_\Omega |Du|^{p(x)} dx < \infty \},
%\]
%and for a given map $\varphi$
%\[
%    \varphi+W^{1, p(x)}_0(\Omega):=
%    \{ u\in W^{1,p(x)}(\Omega)~;~u-\varphi \in W^{1,p(x)}_0(\Omega)\}.
%\]

Let us define \it local minimizers of $\F$ \rm as follows:
\begin{defi}
A function $u \in  W^{1,1}
%{\color{red}{
%_{\mathrm{loc}}(
%}}
\Omega)$
is called to be a local minimizer
of $\F$ if $F(x,Du) \in L^1
%{\color{red}{
%_{\loc}
%}}
(\Omega)$ and satisfies
\[
	\F(u;\supp \varphi) \leq \F(u+\varphi; \supp \varphi),
\]
for any $\varphi \in W^{1,p(x)}_{\loc}(\Omega)$ with compact support in  $\Omega$.
\end{defi}
\noindent
The main result of this paper is the following:
\begin{thm}\la{main}
Assume  that the conditions \pref{cond-p,q-1}, \pref{cond-a} and \pref{cond-p,q-2}
are fulfilled.
Let  $u\in W^{1, \cred1}(\Omega)$
%($\Leftarrow$ ho modificato $p(x)\to 1$ perche` qui sopra abbiamo definito``loacl minimizer" per $u\in W^{1,1}$.\cred\fbox{AT})
 be a local minimizer of $\F$.
Then $u \in C^{1,\gamma}_{\loc} (\Omega)$ for some $\gamma \in (0,1)$.
\end{thm}

%{\color{red}{
%We point out that, since all the results of this paper are of local nature, we write above $C^{1,\gamma}(\Omega)$, in analogy to what is done by classical authors as D. Gilbarg, N. Trudinger, L.C. Evans and others.
%}}

%%%%%%%%%%
%\cred
%~~\\ \fbox{by AT}
\begin{rem}[About the symbols for H\"older spaces]
If we follow the standard textbooks,  Dacorogna \cite{dac15(book)}, Evans \cite{eva10(book)},
Gilberg-Trudinger \cite{giltru98}, etc.,  for $k\in \N, 0<\alpha\leq 1$,
$C^{k,\alpha}(\Omega) $ mean the subspaces of
$C^k(\Omega)$ consisting of functions
whose $k$-th order partial derivatives are locally H\"older continuous.
However, recently many authors (especially ones who study regularity
problems)
write them as $C^{k,\alpha}_\loc (\Omega) $,
and they use $C^{k,\alpha}(\Omega) $  for $C^{k,\alpha}(\bar\Omega)$
(namely, for uniformly H\"older continuous cases).
Anyway, with ``\rm loc" \em there is no doubt of misunderstanding.
%Moreover, 2 referees for this paper suggested us to use
%$C^{k,\alpha}_\loc (\Omega) $.
So, in this paper we follow their usage for H\"older spaces.
%
%Throughout this paper we are assuming $\Omega$ is a bounded open set. So,
%according to the standard textbooks Dacorogna \cite{dac15(book)}, Evans \cite{eva10(book)},
%Gilberg-Trudinger \cite{giltru98}, etc.,  for $k\in \N, 0<\alpha\leq 1$,
%we mean by $C^{k,\alpha}(\Omega)~\left( C^{k,\alpha}(\bar{\Omega})\right)$ the subspaces of
%$C^k(\Omega)~\left( C^k(\bar{\Omega})\right) $ consisting of functions
%whose $k$-th order partial derivatives are locally H\"older continuous
%(uniformly H\"older continuous) with exponent $\alpha$.
%We know that recently many authors write them as
%$C^{k,\alpha}_\loc (\Omega) \left( C^{k,\alpha}(\Omega)\right)$.
%So, please mention that in this paper we use classical(?) notation and that
%$C^{k,\alpha}(\Omega)$ are locally H\"older spaces.
\end{rem}
%\cb
%%%%%%%%%%%%
In order to prove the above theorem, we employ a freezing argument; namely
we consider a frozen functional which is given by freezing the exponents, and
compare a minimizer of the original functional under consideration with that of frozen one.

%%%%%%
%%%%%%
%%%%%%
\section{Preliminary results}
\setcounter{equation}{0}
\setcounter{thm}{0}

In what follows, we use $C$ as generic constants, which may change from line to line,
but does not depend on the crucial quantities.
When we need to specify a constant, we use small letter $c$ with index.

For double phase functional with constant exponents, namely for
\be
	\H(u,D):=\int_D H(x, Du) dx, ~~ H(x,z)= |z|^p+a(x)|z|^q,
\ee
we prepare the following Sobolev-Poincar\'{e} inequality which is a slightly generalised
version of \cite[Theorem 1.6]{colmin15-1} due to Colombo-Mingione.
%%%%%%
%%%%%%
\begin{thm}\la{sobpoi1}
Let $a(x) \in C^{%1 
0
,\beta}(\Omega)$ for some $\beta \in (0,1)$ and $1<p<q$ constants satisfying
\[
	\frac{q}{p} <  1+ \frac{\beta}{n},
\]
and let $\omega\in L^\infty(\R^n)$ with $\omega \geq 0$ and $\int_{B_R} \omega dx=1$
for $B_R\subset \Omega$ with $R\in (0,1)$.
Then, %for any $B_R \Subset \Omega$,
there exists a constant $C$ depending only on $n,~p,~q,~ [a]_{0,\beta},~R^n\|\omega\|_{L^\infty}$ and
$\|Dw\|_{L^p(B_R)}$ and exponents $d_1>1>d_2$ depending only on
$n,~p,~q,~\beta$ such that
\be\la{colmin(1.18)}
	\left( \intmean_{B_R} \left[ H\left(x,\frac{u-\langle u\rangle_\omega}{R}\right)\right]^{d_1} dx \right)^{\frac{1}{d_1}}
	\leq
	C\left( \intmean_{B_R} \left[ H\left(x,Du\right)\right]^{d_2} dx \right)^{\frac{1}{d_2}}
\ee
holds whenever $u \in W^{1,p}(B_R)$,
where
\[
	\langle u\rangle_\omega:=\int_{B_R} u(x)\omega(x) dx.
\]
Note that for the special choice $\omega=|B_R|^{-1}\chi_{B_R}$ we have
\[
	\langle u\rangle_\omega =\intmean_{B_R} u(x) dx.
\]
\end{thm}
\begin{proof}
We can proceed exactly as in the proof of \cite[Theorem 1.6]{colmin15-1} only replacing (3.11) of \cite{colmin15-1}
by
\[
	\frac{|u(x)-\langle u\rangle_\omega|}{R}
	\leq \frac{C}{R} \int_{B_R} \frac{ |Du(y)|}{|x-y|^{n-1} }dy,
\]
which is shown by \cite[Lemma 1.50]{malzie97} (see also the proof of \cite[Theorem 7]{dieett08}).
\end{proof}

From the above theorem, we have the following corollary.　

\begin{cor}\la{sobpoi2}
Assume that all conditions of Theorem \ref{sobpoi1} are satisfied, and let $D$ be a subset of $B_R$ %of
with  positive measure. Then, %for any $B_R \Subset \Omega$,
there exists a constant $C$ depending only on $n,~p,~q,~ [a]_{0,\beta},~R^n/|D|$ and
$\|Du\|_{L^p(B_R)}$ and exponents $d_1>1>d_2$ depending only on
$n,~p,~q,~\beta$ such that
the following inequality holds whenever
$u \in W^{1, p(x)
%\Phi
}(B_R)$ satifies $u\equiv 0$ on $D$:
\be\la{sobpoiineq2}
	\left( \intmean_{B_R} \left[ H\left(x,\frac{u}{R}\right)\right]^{d_1} dx \right)^{\frac{1}{d_1}}
	\leq
	C\left( \intmean_{B_R} \left[ H\left(x,Du\right)\right]^{d_2} dx \right)^{\frac{1}{d_2}}.
\ee

\end{cor}
\begin{proof}
Choosing $\omega$ so that %$\omega = 0$ on $B_R\setminus D\color{black}$
\[
	\omega(x) =
	\begin{cases}
	0 & x \in B_R\setminus D \\
	\frac{1}{|D|} & x \in D
	\end{cases}
\]
and
applying Theorem \ref{sobpoi1}, we get the assertion.
\end{proof}

%\cb

\begin{rem}
In  \cite[Theorem 6.1]{colmin15-1}, and therefore also in the above theorem and
corollary,
the exponent $d_2\in (0,1)$  is chosen so that the following  conditions hold:
\begin{align}
	\frac{q}{p}  <  1+ \frac{\beta d_2}{n} \la{d2-1}\\
	\frac{p}{q(n-1)} +1 > \frac{1}{d_2}. \la{d2-2}
\end{align}
In fact, in \cite{colmin15-1}, they choose a constant $\gamma \in (1,p)$ so that
\[
	\frac{q}{p} < 1+ \frac{\alpha}{\gamma n}~~~\mathrm{and}~~~
	\frac{p+q(n-1)}{\gamma q(n-1)}>1,
\]
(see \cite[(3.6), (3.14)]{colmin15-1}), and put $d_2 = 1/\gamma$.
 Let us mention the that if $d_2$ satisfies \pref{d2-1} and \pref{d2-2}  for
some $q=q_0$ and $p=p_0$, then the same $d_2$ satisfies
these inequalities for any $q$ and $p$ with $q/p \leq q_0/p_0$.
\end{rem}

For any $y\in \Omega$ and $R>0$ with $B_R(x)\subset \Omega$  let us put
\begin{align}
	p_2(y,R) := \sup_{B_R(y)} p(x), ~~  & p_1(y,R) := \inf_{B_R(y)} p(x),\\
	q_2(y,R) := \sup_{B_R(y)} q(x), ~~ & q_1(y,R) := \inf_{B_R(y)} q(x).
\end{align}
We prove interior higher integrability of the gradient of a minimizer, similar results are contained in \cite{HHK}.

\begin{prop}\la{higher-int}
Let $u \in W^{1,p(x)}_\loc (\Omega)$ be a local minimizer of $\F$.
Then, for any compact subset $K \subset \Omega$,
$F(x,Du) \in L^{1+\delta_0}(K)$ and
there exists a positive constant $\delta_0$ and $C$
 depending only on the given data and $K$  such that
\be\la{high-1}
	\left(\intmean_{B_{R/2}(y)} F(x,Du)^{1+\delta_0} dx\right)^{\frac{1}{1+\delta_0}}
	\leq C + C \intmean_{B_R(y)} F(x,Du) dx
\ee
holds  for any $B_R(y)\Subset K$.
\end{prop}
\begin{proof}
Let $K\subset \Omega $ be a compact subset
%%%%%%%%% AT %%%%%%%%%%%%%%%%%%%%
%
%\cred %\fbox{by AT}
and $R_0 \in (0, \dist (K,\partial\Omega))$ a constant
such that
\be\la{def-R0}
	0< R_0^\sigma \leq \frac{p_0}{2^{1+\sigma}[q]_{0,\sigma}}
	\left( 1+\frac{\beta}{n}-\sup_{x\in\Omega}\frac{q(x)}{p(x)}\right).
\ee
For any $x_0 \in \stackrel {\circ }{K}$, put
\be\la{def_kp_0}
	\kappa_0 :=\frac{1}{4}\left( 1+\frac{\beta}{n} -
	\sup_{x \in B_R(x_0)} \frac{q(x)}{p(x)} \right) >0.
\ee
Then, letting $x_-\in \bar{B}_{R_0}(x_0)$ be a such that
$ p(x_-)= p_1(x_0,R_0)$, we have
\begin{align}
	\frac{q_2(x_0,R_0)}{p_1(x_0,R_0)}
	& =  \frac{q(x_-)+\left(q_2(x_0,R_0)-q(x_-)\right)}{p_1(x_0,R_0)}\nn\\
	& \leq \sup_{x\in B_{R_0}(x_0)} \frac{q(x)}{p(x)} + \frac{2^\sigma[q]_{0,\sigma} R_0^\sigma}{p_0}
	\nn\\
	& \leq
	\sup_{x\in B_{R_0}(x_0)}\frac{q(x)}{p(x)} + \frac{1}{2} \left(1+\frac{\beta}{n} -
	\sup_{x\in B_{R_0}(x_0)}\frac{q(x)}{p(x)}\right)\nn\\
	& = \frac{1}{2}\left(1+\frac{\beta}{n} + \sup_{x\in B_{R_0}(x_0)}\frac{q(x)}{p(x)}\right)	
	\leq 1+\frac{\beta}{n}-2\kappa_0
	\la{2.10}
\end{align}
%\cb
%%%%%%%%%%%%%%%%%%%%%%%%%%
The above estimate \pref{2.10} implies that
\be\la{sob-exp-cond}
	q_2(x_0, R_0) < (p_1(x_0, R_0))^\ast=\frac{np_1(x_0,\cb R_0)}{n-p_1(x_0,\cb R_0)}.
\ee

For any $B_R(y) \subset B_{R_0}(x_0)$ with $0<R<1$, and $0<t\le s \le R$, let
$\eta$ be a cut-off function such  that
$\eta \equiv 1$ on $B_t(y)$, $\eta \equiv 0$ outside $B_s(y)$ and
 $|D\eta|\leq \frac{2}{s-t}$.
Put $w:= u-\eta(u-u_R)$, where $u_R =\sintmean_{B_R(y)} u dx$.
Since
\[
	Dw= (1-\eta) Du + (u-u_R) D\eta,
\]
we have
\begin{align*}
	F(x,Dw) \leq ~& c_0\big[\big((1-\eta)|Du|\big)^{p(x)} +\left(|u-u_R|
	|D\eta|\right)^{p(x)}\\
	& +a(x) \big((1-\eta)|Du|\big)^{q(x)} + \left(
%\cbl
|u-u_R| %\mbox{(C'era~troppo~$q(x)$!!)}\cb
	|D\eta|\right)^{q(x)}\big],
\end{align*}
where $c_0$ is a constant depending only on $\max_K q(x)$.
On the other hand, since  $F(x,Du) \in L^1$, we have
\[
	u \in W^{1,p(x)} \subset W^{1,p_1(x_0, R_0)} \subset L^{p_1(x_0,R_0)^\ast} \subset
	L^{p_2(x_0,R_0)} \subset L^{q(x)},
\]
 on $B_{R_0}(x_0)$.
 Thus, mentioning  also that $w=u$ outside $B_s(y)$, we see that $F(x,Dw) \in L^1 (K)$, namely $w$ is an admissible function. In the following part of the proof,  let us abbreviate
\[
	p_i:=p_i (y,R),~~ q_i :=q_i(y,R) ~~(i=1,2).
\]
Then, we have
\begin{align}
	& \int_{B_s(y)} F(x,Du)dx \leq \int_{B_s(y)} F(x,Dw)dx\nn\\
	\leq  ~& c_0\int_{B_s(y)}
	(1-\eta)^{p(x)} \big(|Du|^{p(x)} + a(x)|Du|^{ q(x)}\big)dx\nn\\
	& ~~ +c_0 \int_{B_s(y)} \left[
	\left|\frac{u-u_R}{s-t}\right|^{p(x)}+ a(x) \left|\frac{u-u_R}{s-t}\right|^{q(x)}\right] dx\nn\\
	\leq ~& c_0\int_{B_s(y)\setminus B_t(y)} F(x,Du) dx
	+\frac{c_0}{(s-t)^{p_2}}\int_{B_s(y)} |u-u_R|^{p(x)}\nn\\
	& ~~~~~ +
	\frac{c_0}{(s-t)^{q_2}}\int_{B_s(y)} a(x) |u-u_R|^{q(x)}dx
\end{align}
We can use \it hole-filling method. \rm
Add $c_0  \int_{B_s(y)\setminus B_t(y)} F(x,Du) dx$ to the both side and divide them
by $c_0+1$, then we get
\begin{align}
	& \int_{B_t (y)} F(x,Du)dx \nn\\
	\le ~&\frac{c_0}{c_0+1}\left( \int_{B_s (y)} F(x,Du)dx
	+ \frac{1}{(s-t)^{p_2}}\int_{B_s(y)} |u-u_R|^{p(x)}dx\right.\nn\\
	&~~~~ \left. + \frac{1}{(s-t)^{q_2}}\int_{B_s(y)} a(x) |u-u_R|^{q(x)}dx\right).
\end{align}
Using an iteration lemma \cite[Lemma 6.1]{giu03}, we see, for some constant
$C=C(c_0, p_2, q_2)$, that
\begin{align}
	& \int_{B_t (y)} F(x,Du)dx \nn\\
 	\leq ~ &  \frac{C}{(s-t)^{p_2}}\int_{B_s(y)} |u-u_R|^{p(x)}
	 + \frac{C}{(s-t)^{q_2}}\int_{B_s(y)} a(x) |u-u_R|^{q(x)}dx.
\end{align}
Putting $s=R$ and $t=R/2$, we have
\begin{align}
	& \int_{B_{\frac{R}{2}} (y)} F(x,Du)dx \nn\\
 	\leq ~ & \frac{C}{R^{p_2}}\int_{B_R(y)} |u-u_R|^{p(x)}
	 + \frac{C}{R^{q_2}}\int_{B_R(y)} a(x) |u-u_R|^{q(x)}dx\nn\\
	 \leq ~& CR^{p_1-p_2} \int_{B_R(y)} \left| \frac{u-u_R}{R}\right|^{p(x)}dx
	 +CR^{q_1-q_2} \int_{B_R(y)} a(x) \left| \frac{u-u_R}{R}\right|^{q(x)}dx\nn\\
	 \leq ~&  CR^{p_1-p_2} \int_{B_R(y)}
	 \left( 1+\left| \frac{u-u_R}{R}\right|\right)^{p_2}dx\nn\\
	 & ~~~~~~~~~ ~~~~~ ~~~~
	 +CR^{q_1-q_2} \int_{B_R(y)} \left( 1+ a(x)^{\frac{1}{q(x)}} \left| \frac{u-u_R}{R}\right|		\right)^{q_2}dx.
	 \la{est14}
\end{align}
Since $R^{p_1-p_2}$ and $R^{q_1-q_2}$ are bounded because of the H\"older continuity
of exponents $p(x)$ and $q(x)$,
putting
\[
	\tilde{a}(x):= \left(a(x)\right)^{\frac{q_2}{q(x)}},
\]
from \pref{est14}, we obtain the estimate
\begin{align}
& \int_{B_{\frac{R}{2}} (y)} F(x,Du)dx \nn\\
	 \leq ~& CR^n +
	 CR^n \intmean_{B_R (y)} \left(\left| \frac{u-u_R}{R}\right|^{p_2}dx
	 + \tilde{a}(x) \left| \frac{u-u_R}{R}\right|^{q_2}\right)dx\nn\\
	 =: ~& I +I\!I .
	 \la{est15}
\end{align}
In order to get the boundedness of  $R^{p_1-p_2}$ and $R^{q_1-q_2}$
the so-called ``log-\it H\"older continuity" \rm (see \cite[section 4.1]{DHHR}) is sufficient.
On the other hand by virtue of the H\"older continuity of $q(\cdot)$, we have that $\tilde{a} \in C^{0,\beta}~~(\beta =\min\{\alpha,\sigma\})$.
Let $d_2\in (0,1)$ be a constant satisfying \pref{d2-1} and \pref{d2-2} for $\beta= \min\{\alpha,\sigma\}$, $q=q_2(x_0,R_0)$ and $p=p_1(x_0, R_0)$.
Then, for any $B_R(y) \subset B_{R_0}(x_0)$, this $d_2$ satisfy \pref{d2-1} and \pref{d2-2} with $q=q_2(y,R)$ and $p=p_2(y,R)$.

By Theorem \ref{sobpoi1}, we can estimate $I\!I$ as follows.
\begin{align}
	I\!I \leq ~& C R^n
\left( \intmean_{B_R(y)}\left( |Du|^{p_2} + \tilde{a}(x)
	|Du|^{q_2}\right)^{d_2} dx \right)^{\frac{1}{d_2}} \nn\\
	\leq ~& C
R^n \left( \intmean_{B_R(y)} |Du|^{d_2 p_2} dx\right)^{\frac{1}{d_2}}
	+  C R^n  \cb \left(\intmean_{B_R(y)} \left(a(x)^{\frac{1}{q(x)}}
	|Du|\right)^{d_2 q_2} dx \right)^{\frac{1}{d_2}}.\la{est20}
\end{align}
%As mentioned above,
%\cm
%scriviamo solo "As mentioned, " ?
%\cb
%\cbl
%Secondo me qui ci bisogno ``above".
%Per evitare doppio ``above", scriviamo ``
As mentioned above, \pref{est20} holds for
%...
%\cb
%the above estimate hold
 for any $B_R(y) \subset B_{R_0}(x_0)$
with same $d_2$.  Now, take $R>0$ sufficiently small
so that
\[
	d_2 p_2(y,R) < p_1(y,R) ~~\text{and}~~ d_2q_2(y,R) < q_1(y,R),
\]
and let $\theta \in (d_2,1)$ be a constant satisfying
\be
	d_2 p_2(y,R) < \theta p_1(y,R) ~~\text{and}~~ d_2q_2(y,R) < \theta q_1(y,R).
\ee
Then, using H\"older inequality, we can estimate the first term of the right hand side of
\pref{est20}  as follows.

%{\color{red}{
%ALESSANDRA: NON MI SEMBRA CHE IL REFEREE SCRIVA IL RFERIMENTO CORRETTO, CIOE' PAGINA 6 %LINES 19 e 21. FORSE, MA NON SONO SICURA, IL REFEREE POTREBBE INTENDERE QUI, pg. 9 LINES 19 (E %RIGHE PRECEDENTI DELLA FORMULA) E 23.
%}}

%%%%%%%%
%\cred \fbox{AT}  Ancora lasciamo questa parte come alla vecchia versione. Perche` ancora non e` sicuro
%il referee ha indicato qui o no.
%Fare $B_R(y) \Longrightarrow B_{R_0}({x_0})$ soltanto qui e` molto strano. Se fare questo cambiamento,
%dobbiamo modificare tanti parti. Comunque aspettiamo la risposta dal referee.
%\cb
%%%%%%%%%%
\begin{align}\la{est22}
	 & \left(\intmean_{
%{\color{red}{B_{R_0}({x_0})}}
B_R(y)
} |Du|^{d_2 p_2} dx\right)^{\frac{1}{d_2}}
	 \leq  \left(\intmean_{
%{\color{red}{B_{R_0}({x_0})}}
B_R(y)
} |Du|^{\theta p_1} dx\right)^{\frac{p_2}{\theta p_1}}\nn\\
	 = ~&   \left(\intmean_{
%{\color{red}{B_{R_0}({x_0})}}
B_R(y)
} |Du|^{\theta p_1} dx\right)^{\frac{p_2-p_1}{\theta p_1}}
	 \cdot \left(\intmean_{
%{\color{red}{B_{R_0}({x_0})}}
B_R(y)
} |Du|^{\theta p_1} dx\right)^{\frac{1}{\theta}} \nn\\
	 \leq ~& \left(\intmean_{
%{\color{red}{B_{R_0}({x_0})}}
B_R(y)
}( 1+|Du|^{p(x)}) dx\right)^{\frac{p_2-p_1}{\theta p_1}}
	 \cdot \left(\intmean_{
%{\color{red}{B_{R_0}({x_0})}}
B_R(y)
} \left(1+ |Du|^{\theta p_1} \right) dx\right)^{\frac{1}{\theta}} .
\end{align}

Since,
\[
	\int_{
%{\color{red}{B_{R_0}({x_0})}}
B_R(y)
} |Du|^{p(x)} dx \leq \F(u,B_R(y)) \leq \F(u,K)\cb
\]
and $u$ locally minimizes $\F$, $\int_{B_R(y)} |Du|^{p(x)} dx$ is bounded.
On the other hand, as mentioned after \pref{est14}, $R^{-(p_2-p_1)}$ is bounded.
%On the other hand, by virtue of the H\"older continuity of $p(x)$,
%it is easy to see that $R^{-(p_2(R)-p_1(R))}$ is bounded.
%\cb
So,
there exists a constant $c_1=c_1 (\F(u,K), p(x), d_2, n, \theta )$
\begin{align*}
	\left(\intmean_{B_R(y)} |Du|^{p(x)} dx\right)^{\frac{p_2-p_1}{\theta p_1}}
	\leq ~&
 (\omega_n R^n)^{\frac{-(p_2-p_1)}{\theta p_1}} \F(u,K)^{\frac{p_2-p_1}{\theta p_1}} \\
\leq ~ & c_1 (\F(u,K), p(x), d_2, n, \theta),
\end{align*}
where $\omega_n$ denotes the volume of a $n$-dimensional unit ball.
Thus, from \pref{est22}
we obtain for some positive constant $c_2 = c_2(c_1, \theta)$
\be\la{est23}
	\left(\intmean_{B_R(y)} |Du|^{d_2 p_2} dx\right)^{\frac{1}{d_2}} \leq
	c_2+c_2\left(\intmean_{B_R(y)} |Du|^{\theta p(x)} dx\right)^{\frac{1}{\theta}}.
\ee

Similarly, we can estimate the second term of the left hand side of \pref{est20}
as follows.
\begin{align}
	 &  \left( \intmean_{B_R(y)} \left(a(x)^{\frac{1}{q(x)}}
	|Du|\right)^{d_2 q_2} dx \right)^{\frac{1}{d_2}}
	\leq \left( \intmean_{B_R(y)} \left(a(x)^{\frac{1}{q(x)}}
	|Du|\right)^{\theta q_1} dx \right)^\frac{q_2}{\theta q_1}\nn\\
	\leq ~ & \left( \intmean_{B_R(y)} \left(a(x)^{\frac{1}{q(x)}}
	|Du|\right)^{\theta q_1} dx \right)^\frac{q_2-q_1}{\theta q_1}
	\left( \intmean_{B_R(y)} \left(a(x)^{\frac{1}{q(x)}}
	|Du|\right)^{\theta q_1} dx \right)^\frac{1}{\theta}
	\nn\\
	\leq ~ & \left( \intmean_{B_R(y)} \left( 1+ \left(a(x)^{\frac{1}{q(x)}}
	|Du|\right)^{q(x)} \right) dx \right)^\frac{q_2-q_1}{\theta q_1}\nn\\
	& ~~ ~~  ~~ ~~  \cdot
	\left( \intmean_{B_R(y)}\left(1+  \left(a(x)^{\frac{1}{q(x)}}
	|Du|\right)^{\theta q(x)} \right)dx \right)^\frac{1}{\theta}.
	\la{est24}
\end{align}
As above, using local minimality of $u$ and the fact that
$R^{-(q_2-q_1)}$ is bounded, we have for a positive constant $c_3=c_3(\F(u,K), q(x), d_2, n, \theta)$
\be
	\left( \intmean_{B_R(y)} \left(a(x)^{\frac{1}{q(x)}}
	|Du|\right)^{d_2 q_2} dx \right)^{\frac{q_2 - q_1}{\theta q_1} }
	\leq c_3 (\F(u,K), q(x), d_2, n, \theta).
\ee
Thus, we obtain for some positive constant $c_4=c_4(c_3, \theta)$
\begin{align}
	\left( \intmean_{B_R(y)} \left(a(x)^{\frac{1}{q(x)}}
	|Du|\right)^{d_2 q_2} dx \right)^{\frac{1}{d_2}}
	\leq c_4 + c_4 \left( \intmean_{B_R(y)} \left(a(x)^{\frac{1}{q(x)}}
	|Du|\right)^{\theta q(x)} dx \right)^\frac{1}{\theta}.\la{est26}
\end{align}
Combining \pref{est15}, \pref{est20}, \pref{est23} and \pref{est26},
we see that there exists a constant $C$ depending on the given data
 and $\F(u,K)$  such that
\be
	\intmean_{B_{\frac{R}{2}} (y)}F(x,Du) dx
	\leq C + C \left(\intmean_{B_R(y)} F(x,Du)^\theta dx \right)^{\frac{1}{\theta}}
\ee
for any $B_R(y) \subset B_{R_0} \subset K \Subset \Omega$.
Now, by virtue of the \it reverse H\"older inequality with increasing domain \rm
due to Giaquinta-Modica \cite{giamod79}, we get the assertion.
\end{proof}

For $\delta_0$ determined in Proposition \ref{higher-int}, in what follows, we always take $R>0$ sufficiently small so that
\be\la{choice of R}
	\left(1+ \frac{\delta_0}{2}\right)p_2(y,R) \leq (1+\delta_0)p_1(y,R)~~
	\text{and}~~\left(1+ \frac{\delta_0}{2}\right)q_2(y,R) \leq (1+\delta_0)q_1(y,R).
\ee

%%%%%%%%
%%%%%%%%
We need also higher integrability results on the neighborhood  of the
boundary.
Let us use the following notation: for $T>0$ we put
\begin{align*}
     &B_T := B_T(0), ~~B^+_T:= \{ x\in \R^n ~;~|x|<T, ~x^n >0\}, \\
     &\Gamma_T := \{ x\in \R^n ~;~ |x|<T, ~x^n=0\}, %\\
%     &\partial^+ B^+_T := \partial B^+_T\setminus \Gamma_T. \nn
\end{align*}
We say ``$f=g$ on $\Gamma_T$" when for any $\eta \in C^\infty_0(B_T)$
we have $(f-g)\eta \in W^{1,1}_0(B^+_T)$.
For $y\in B_T$, we write
\[
	\Omega_r:= B_r(y) \cap B^+_T.
\]
Then, we have the following proposition on the higher integrability near the boundary,
independently proved in \cite[Lemma 5]{DO} , see also \cite[Lemma 5]{DM} for the manifold constrained case.
%\cm
\begin{prop}\la{higher-bdry}\cb
Let $a(x)$, $q$ and $p$ satisfy the same conditions in Theorem \ref{sobpoi1} and let
for $A \subset B^+_T$
\[
	\H(w,A) :=\int_A H(x,w) dx, ~~H(x,z) := |z|^p+a(x)|z|^q.
\]
$u \in W^{1,p}(B^+_T)$ be a given function with
\[
	\int_{B^+_T} \left(|Du|^p + a(x) |Du|^q \right)^{1+\delta_0}dx < \infty,
\]
for some $\delta_0 >$.
Assume that $v\in W^{1,p} (B^+(T))$ be a local minimizer of $\H$ in the class
\[
	\{w \in W^{1,p}(B^+_T)~;~ u=w~~\text{on}~~\Gamma_T\}
\]
Then, for any $S \in (0,T)$, there exists a constants $\delta \in (0,\delta_0)$
and $C>0$ such that for any
$y \in B^+_S$ and $R\in (0,T-S)$ we have
\[
\left( \intmean_{\Omega_{R/2} }\left( H(x, Dv)\right)^{1 +\delta} dx \right)^{\frac{1}{1+\delta}} \leq
C \intmean_{\Omega_{R} } H(x,Dv) dx +
C \left( \intmean_{\Omega_R} \left( H(x, Du)\right)^{1 +\delta} dx \right)^{\frac{1}{1+\delta}} .%\leq
\]

\end{prop}
%\cb
\begin{proof}
%\cm
For convenience, we extend $u, v, Du, Dv$ to be zero in $B_T \setminus B^+_T$.
Of course, because extended $u, v$ may have discontinuity  on $\Gamma_T$,
they are not always in $W^{1, p}_\loc (B_T)$, and therefore
$Du,  Dv$ do not necessarily coincide with distributional derivatives of $u, v$ on $B(T)$.
On the other hand, since $u=v$ on $\Gamma(T)$, $u-v$ is
in the class $W^{1,p}(B(S))$ and $Du-Dv$ can be regarded as
the weak derivatives of $u-v$ on $B(S)$ for any $S < T$.

Let $R$ be a positive constant satisfying $R \le (T-S)/2$.
For $x_0\in B^+_{S}$,
we treat the two cases $x_0^n \leq \frac{3}{4}R$ and $x_0^n > \frac{3}{4}R$ separately.
%Since $S$ is fixed, from now on we shall the shorten notation
%$\Omega _{r}(x) := B_r(x) \cap B^+_{S}$.

\medskip
 \noindent
 \bf Case 1. \rm
Suppose that $x_0^n\leq \frac{3}{4}R$.
Take radii $s, t$ so that  $0<R/2 \leq t<s \leq R$
%\cb
and choose a $\eta\in C^\infty_0(B_T)$ % \in C^\infty_0 (B_s)$
such that $0\leq \eta \leq 1$, $\eta \equiv 1$ on $B_t$,
 $\supp ~\eta \subset B_s$ and $|D\eta| \leq 2/(s-t)$.
%\[
 %   \Omega_R := \Omega(x, R) = \Omega \cap B(x,R),
%\]
%\cm
Defining
%\cb
\[
	\varphi := \eta (v-u),
\]
we see that $\varphi \in W^{1,1}_0 (B^+_T)$ with  $\supp~ \varphi \subset B_s$,
%\cm
 and that
%\cb
\[
    D(v-\varphi) = (1-\eta)Dv - (v-u)D\eta + \eta Du.
\]
Then, by virtue of the minimality of $v$,
%\cm
  %  we have lo eliminerei da qui
%\cb
    for a positive constant $c_4$ depending only on $q$,
%\cm
    we have % (per metterlo qui)
%\cb
\begin{align*}
	\int_{\Omega_t}  H(x,Dv) dx \leq  ~ &
	 \int_{\Omega_s}  H(x,Dv)  dx \leq
	\int_{\Omega_s} H(x,D(v- \varphi))  dx \\
	=  ~& \int_{\Omega_s} \left( |D(v- \varphi)|^p + a(x)|D(v- \varphi)|^q \right) dx \\
	\leq ~& c_4 \int_{\Omega_s \setminus \Omega_t} \left( |D v|^p + a(x)|D v|^q \right) dx
	+ c_4 \int_{\Omega_s} \left( |D u|^p +  a(x)|D u|^q \right) dx \\
	& ~~ ~~ ~~ +  c_4 \int_{\Omega_s} \left(\left( \frac{2}{s - t} \right)^p |v - u|^p +
	a(x)  \left(\frac{2}{s - t} \right)^q |v - u|^q \right) dx \\
	 \leq ~&  c_4 \int_{\Omega_s \setminus \Omega_t} \left( |D v|^p + a(x)|D v|^q \right) dx
	  + c_4 \int_{\Omega_s} \left( |D u|^p + a(x)|D u|^q \right) dx \\
	& ~~~ ~~~ ~~~ +  c_4  \left( \frac{2}{s - t} \right)^p \int_{\Omega_s}  |v - u|^p dx
	+ c_4 \left( \frac{2}{s - t} \right)^q \int_{\Omega_s} a(x)   |v - u|^q
	dx.
\end{align*}
 Now, we use {\it the hole filling method}
 as in the proof of Proposition \ref{higher-int}.
Namely,
 adding
\[
	c_4 \int_{\Omega_t} \left( |D v|^p + a(x)|D v|^q \right) dx
\]
%\cm
 and dividing both side by $c_4+1$, we obtain
\begin{align*}
	& \int_{\Omega_t} H(x, Dv)  dx \\
	\leq ~& \frac{c_4}{c_4 +1} \left( \int_{\Omega_s} H(x, Dv) dx
	+ \int_{\Omega_s} H(x, Du) dx \right.\\
	& ~~ ~~ ~~ + \left. \frac{1}{(s - t)^p} \int_{\Omega_s}  |v - u|^p dx +
	\frac{1}{(s - t)^q} \int_{\Omega_s} a(x)   |v - u|^q dx \right),
\end{align*}
Using %\cite[Lemma 2.2]{colmin15-1} due to Colombo-Mingione,
the iteration lemma \cite[Lemma 6.1]{giu03},
%\cb
we get
%\cm
 for some constant $C= C(c_4, p, q)$
%\cb
\begin{align*}
	\int_{\Omega_t} H(x, Dv) dx \leq ~ & C \int_{\Omega_s} H(x, Du) dx \\
	& ~~ ~~ ~~+ \frac{C}{(s - t)^p} \int_{\Omega_s}  |v - u|^p dx
	+ \frac{C}{(s - t)^q} \int_{\Omega_s} a(x)   |v - u|^q dx.
\end{align*}
Putting  $t=R/2$ and  $s=R$, we have
\[
	\int_{\Omega_{R/2}} H(x, Dv) dx
	\leq C \int_{\Omega_{R}} H\left( x, \frac{v-u}{R}\right) dx
	+ C \int_{\Omega_{R}} H(x, Du) dx.
\]
%\cb
 Let us now consider the mean integral in all the terms, we obtain
\[
	\intmean_{\Omega_{R/2}} H(x,Dv) dx \leq
 	C \intmean_{\Omega_{R}}  H(x,Du) dx + C \intmean_{\Omega_{R}}  H\left(x,\frac{ v-u }{R}
	\right) dx.
\]
%\cm
Since we are assuming that $x_0^n\leq \frac{3}{4}R$
we can apply Corollary \ref{sobpoi2} with a constant independent on $R$
for the last term in the right hand side and get
\[
	\intmean_{\Omega_{R/2}} H(x,Dv) dx
	\leq C \intmean_{\Omega_{R}}  H(x,Du) dx + C \
 	\left( \intmean_{\Omega_{R}} (H(x, D(v-u)))^{d_2} dx \right)^{\frac{1}{d_2}}.
\]
 %\cb
%Then, we have
%\[
%\intmean_{\Omega_{\frac{R}{2} } } \left( H(x,Dv) \right) dx \leq
% C \intmean_{\Omega_R}  H(x,Du) dx + C \left( \intmean_{\Omega_R} \left( H(x, D(v-u)\right)^{d_2} dx \right)^{d_2} \leq
%\]
%\cb
Taking into consideration that $d_2 \,<\, 1 $ we share in the last term $Dv$ and $Du$, apply H\"older inequality
%\cm
 for the integral of $H(x,Du)^{d_2}$,  and obtain
%\[
%\leq
%C \intmean_{\Omega_R}  H(x,Du) dx + \left( \intmean_{\Omega_R} \left( H(x, Dv)\right)^{d_2} dx \right)^{d_2} +
%C \left( \intmean_{\Omega_R} \left( H(x, Du)\right)^{d_2} dx \right)^{d_2} \leq
%\]
%\[
%\leq
%C \intmean_{\Omega_R}  H(x,Du) dx + \left( \intmean_{\Omega_R} \left( H(x, Dv)\right)^{d_2} dx \right)^{d_2}
%\]
%and finally we get
% \cb
\be\la{est_case1}
	\intmean_{\Omega_{R/2}} H(x,Dv) dx \leq
	C \intmean_{\Omega_{R} } H(x,Du) dx +
	C \left( \intmean_{\Omega_{R}} \left( H(x, Dv)\right)^{d_2} dx \right)^{\frac{1}{d_2}} .
\ee

\noindent
%\cm
\bf Case 2. \rm
Let us deal with the case that $x_0^n>\frac{3}{4}R$.
In this case, since $B_{3R/4}(x_0) \Subset B^+_T$,
we can proceed as in \cite[9. Proof of Theorem 1.1:(1.8)]{colmin15-1},  slightly
modifying the radii, to get
\begin{align}
	& \intmean_{\Omega_{R/2}} H(x,Dv)dx = \intmean_{B_{R/2}} H(x,Dv)dx \nn\\
	\leq ~& C\left( \intmean_{B_{3R/4}} \left( H(x, Dv)\right)^{d_2} dx \right)^{\frac{1}{d_2}}
	\leq C^\prime  \left( \intmean_{\Omega_{R}} \left( H(x, Dv)\right)^{d_2} dx \right)^{\frac{1}{d_2}}.
	\la{est_case2}
\end{align}

Thus, we see that \pref{est_case1} holds for every $0<R<(S-T)/2$.
Now, the reverse H\"older inequality allows us to obtain
\[
\left( \intmean_{\Omega_R} \left( H(x, Dv)\right)^{1 +\delta} dx \right)^{\frac{1}{1+\delta}} \leq
C \intmean_{\Omega_{\frac{R}{2} } } H(x,Dv) dx +
C \left( \intmean_{\Omega_R} \left( H(x, Du)\right)^{1 +\delta} dx \right)^{\frac{1}{1+\delta}}.
\]
\end{proof}

By virtue of \cite[Theorem 1.1]{colmin15-1} and Proposition \ref{higher-bdry}, we have the following
global higher integrability for functions which minimize $\H$ with Dirichlet boundary condition.
\begin{cor}\la{higher-global}
Let $a(x)$, $q$ and $p$ satisfy the same conditions in Theorem \ref{sobpoi1} and $\delta_2\in (0,1)$
be a some constant.
Assume that $u \in W^{1,(1+\delta_1)p}(B_R(y))$ be a given function with
\[
	\int_{B_R(y)} H(x,Du)^{1+\delta_1} dx := \int_{B_R(y)} \left( |Du|^p+a(x)|Dv|^q\right)^{1+\delta_1} dx\leq C
\]
for some constant $C>0$.
Let $v \in W^{1,p}(B_R(y))$ be a minimizer of
\[
	\H(w, B_R(y):= \int_{B_R(y)}
	H(x,Dw) dx
\]
in the class
\[
	u+W^{1,p}_0(B_R(y))=\{ w \in W^{1,p}(B_R(y))~;~ u-w \in W^{1,p}_0(B_R(x_0))\}.
\]
Then, for some $\delta_2 \in (0,\delta_1)$ and for any $\delta_3 \in (0,\delta_2)$,
we have $H(x,Dv) \in L^{1+\delta}(B_R(y))$ and
\be\la{2.28}
	\int_{B_R} \left( H(x,Dv) \right)^{1+\delta_3} dx \leq C
	\int_{B_R} \left( H(x,Du) \right)^{1+\delta_3} dx.
\ee
\end{cor}
\begin{proof}
From
\cite[Theorem 1.1]{colmin15-1}, Proposition
\ref{higher-bdry} and  covering argument, we have
\[
	\left( \intmean_{B_R} \left( H(x, Dv)\right)^{1 +\delta} dx \right)^{\frac{1}{1+\delta}} \leq
	C \intmean_{B_R } H(x,Dv)  dx +
	C \left( \intmean_{B_R} \left( H(x, Du)\right)^{1 +\delta} dx \right)^{\frac{1}{1+\delta}}
\]
and then, by the minimality of $v$,
\[
	\left( \intmean_{B_R} \left( H(x, Dv)\right)^{1 +\delta} dx \right)^{\frac{1}{1+\delta}}
	\leq C \intmean_{B_R} H(x,Du) dx +
	C \left( \intmean_{B_R} \left( H(x, Du)\right)^{1 +\delta} dx \right)^{\frac{1}{1+\delta}}
\]
Once again we use the H\"older inequality for the first term of the
right-hand side that  gives us the assertion.%\cb%concluding inequality
%\[
%\left( \intmean_{B_R%\Omega_R
%} \left( H(x, Dv)\right)^{1 +\delta} dx \right)^{1 +\delta} \leq
%C
%\left( \intmean_{B_R %\Omega_R
%} \left( H(x, Du)\right)^{1 +\delta} dx \right)^{1 +\delta}.
%\]
\end{proof}

\section{Proof of the main theorem}
\setcounter{equation}{0}
%\begin{thm}\la{C0alpha}
%Let $u$ be a local minimiser of $\F$, then
%$u \in C^{0,\alpha}(\Omega)$ for any $\alpha\in (0,1)$
%\end{thm}
 In this section we prove Theorem \ref{main}.
 We employ the so-called {\it direct approach}, namely we consider a \it frozen functional \rm
 for which the regularity theory has been established in \cite{colmin15-1} and compare
 a local minimizer of the frozen functional with $u$ under consideration.

For a constant $p>1$, let us define the auxiliary vector field $V_p: \R^n \to \R^n$
as
\be\la{def-V_p}
	V_p(z) := |z|^{p-2}z.
\ee
Let mention that $V_p$ satisfies
\be\la{V_p}
	|V_p(z)|^2 = |z|^p~~\text{and}~~
	|V_p(z_1) - V_p(z_2)| \approx (|z_1|+|z_2|)^{\frac{p-2}{2}} |z_1-z_2|.
\ee

\medskip
\noindent
{\it Proof of Theorem \ref{main}.}  \, We divide the proof into two parts.
We prove the H\"older continuity of $u$ in {\bf Part 1}, and
of the gradient $Du$ in {\bf Part 2}.

\medskip
\noindent
{\bf  Part 1.}
%\cred\fbox{AT}(Vorrei cambiare qui un po.)\cb
Let $K$ and  $B_{R_0}(x_0)$, %$B_R(y)$, $p_i$ and $q_i$
are as in the Proposition \ref{higher-int}.
%\cred
For  $B_R(y) \subset B_{2R}(y)\subset B_{R_0}(x_0)$,
let us define $p_i$ and $q_i$ as in the Proposition \ref{higher-int}.
%\cb
We define a \it frozen functional \rm $\F_0$ as
\begin{align}
	F_0(x,z):= ~&|z|^{p_2}+a(x)^{\frac{q_2}{q(x)}}|z|^{q_2} \\
	\F_0(w,D)= ~& \int_{B_R(y)}F_0(x,Dw)dx \la{def_F_0}.
\end{align}
In what follows, let us abbreviate $\tilde{a}(x)=\left( a(x)\right)^{\frac{q_2}{q(x)}}$
as in the proof of Proposition \ref{higher-int}.

Let $v\in W^{p_2}(B_R(y))$ be a minimizer of $\F_0$ in the class
\[
	u+W^{p_2}_0(B_R(y)):=\{ w \in W^{p_2}(B_R(y))~;~
	w-u \in W^{p_2}_0(B_R(y))\}.
\]
Then, by \cite[Theorem1.3]{colmin15-1}, for any $\gamma \in (0,1)$ there exists
a constant $C>0$ dependent on
$
n, p_2, q_2, \lambda, \Lambda, [ {\tilde{a}} ]_{0, \beta}, \| \tilde{a}\|_\infty, \| Dv\|_{L^{p_2}(B_R(y))}
%\text{(Ho modificato questo ultima termine. \fbox{AT}})
$ and $ \gamma
$
 such that
\be\la{morrey-0}
	\int_{B_\rho(y)} F_0(x,Dv) dx
	\leq C \left( \frac{\rho}{R}\right)^{n-\gamma} \int_{B_R(y)} F_0(x,Dv) dx
	\leq C \left( \frac{\rho}{R}\right)^{n-\gamma}\int_{B_R(y)} F_0(x,Du) dx,
\ee	
where we used the minimality of $v$.
%%%%%%%%%%%%%%%%%%
%\cred%\fbox{AT}(Ho messo un po di spiegazione qui.)
Here, we mention that by the coercivity of the functional and the minimality of $v$
we have the following:
\be\la{2019-1}
	\| Dv\|^{p_2}_{L^{p_2}(B_R(y))} \leq \F_0(v,B_R(y)) \leq \F_0(u, B_R(y)).
%	\leq \int_{B_R(y)} \left( |Du|^{p_2} + \sup \tilde{a}(x)|Du|^{q_2}\right) dx
\ee
On the other hand, since we are taking $R>0$ sufficiently small so that \pref{choice of R} holds,
there exists a constant $C(p_2,q_2)>0$ such that
\be\la{2019-2}
	F_0(x, \xi) \leq C(p_2, q_2) (1+F(x,\xi))^{1+\delta_0}
\ee
holds for any $(x,\xi)\in B_R(y) \times \R^{nN}$.
Now, by virtue of above 2 estimates and Proposition \ref{higher-int}, we can see,
for a constant $C>0$ depending only on the given data on the functional, that
\be\la{2019-3}
	\| Dv\|^{p_2}_{L^{p_2}(B_R(y))} \leq \F_0(v,B_R(y)) \leq C\left(1+\F(u, K)\right)^{1+\delta}.
\ee
Because of the local minimality of $u$, the last quantity is finite.
Consequently, we can regard the constant in \pref{morrey-0} is a constant
depending only on given data and $\F(u,K)$.
%\cb
%%%%%%%%%%%%%%%%

For further convenience, let us mention that from \pref{morrey-0}, is nothing to see that
\begin{align}
	\int_{B_\rho(y)} (1+F_0(x,Dv)) dx
	\leq ~ & C \left( \frac{\rho}{R}\right)^{n-\gamma} \int_{B_R(y)} (1+F_0(x,Dv)) dx\nn\\
	\leq ~ &  C \left( \frac{\rho}{R}\right)^{n-\gamma}\int_{B_R(y)} (1+F_0(x,Du)) dx.\la{morrey-1}
\end{align}

Let us compare $Du$ and $Dv$.
Mentioning the elementary equality for a twice differentiable function
\[
	f(1)-f(0)= f^\prime (0)+\int_0^1(1-t) f^{\prime\prime}(t) dt,
\]
as \cite[(9)]{cosmin99}, and using the fact that $v$ satisfies the Euler-Lagrange equation
of $\F_0$, we can see that
\begin{align}
	& \F_0(u)-\F_0(v)\nn\\
	=~&  \int_{B_R(y)} \frac{d}{dt} F_0(x,tDu-(1-t)Dv)\big|_{t=0} dx\nn\\
	& ~~~ ~~~ ~~~
	+ \int_{B_R(y)} dx \int_0^1 (1-t)\frac{d^2}{dt^2}F_0(x,tDu+(1-t)Dv)dt\nn\\
	= ~& \int_{B_R(y)} D_{z} F_0(x, Dv)(Du-Dv)\nn\\
	& + \int_{B_R(y)} dx \int_0^1 (1-t)D_zD_z F_0(x,tDu+(1-t)Dv)(Du-Dv)(Du-Dv)  dt \nn\\
	\geq ~ &  C\int_{B_R(y)} dx \int_0^1(1-t)\left[ |tDu+(1-t)Dv|^{p_2- 2 }
	\right.\nn\\
	& ~~ ~~ ~~ ~~ ~~ ~~  ~~ \left.
	+\tilde{a}(x) |tDu+(1-t)Dv|^{q_2- 2} \right] |Du-Dv|^2 dt\nn\\
	\geq ~ & C\int_{B_R(y)}\left(|Du|^{p_2-2}+ |Dv|^{p_2-2}\right)
	|Du-Dv|^2 dx\nn\\
	& ~~ ~~ ~~ ~~
	+ \int_{B_R(y)}\tilde{a}(x) \left(|Du|^{q_2-2}+ |Dv|^{q_2-2}\right)
	|Du-Dv|^2 dx.\la{est33}
\end{align}
On the other hand, by the minimality of $v$, we have
\be
	\F_0(u) - \F_0(v) \leq
	\F_0(u) -\F(u,B_R(y)) + \F(v,B_R(y))-\F_0(v).\la{est34}
\ee

 Since we are assuming $p(x), q(x) \in C^{0,\sigma}$,
using the inequality \cite[(7)]{cosmin99}, we can see that,
for any $\vep \in (0,1)$, there exists a positive constant $C$ such that
\begin{align}
	& \F_0(u) -\F(u,B_R(y))\nn\\
	\leq ~& \int_{B_R(y)} \left[\left( |Du|^{p_2} -|Du|^{p(x)}\right) +
	\left( \left(a(x)^{\frac{1}{q(x)}}|Du|\right)^{q_2}- \left(a(x)^{\frac{1}{q(x)}} |Du|\right)^{q(x)}\right)\right] dx \nn\\
%	\leq & \int_{B_R(y)} \left[\left( |Du|^{p_2} -|Du|^{p(x)}\right) +
%	\left( \tilde{a}(x)|Du|^{q_2}- a(x) |Du|^{q(x)}\right)\right] dx \nn\\
	\leq ~& C(\vep)R^\sigma \int_{B_R(y)} \left( 1+|Du|^{(1+\vep)p_2}\right) dx\nn\\
	& ~~ ~~ ~~ ~~ + C(\vep) R^\sigma \int_{B_R(y)}
	\left( 1+\left(a(x)^{\frac{1}{q(x)}}|Du|\right)^{(1+\vep)q_2}\right) dx\nn\\
	\leq ~ &CR^{n+\sigma}+ C(\vep)R^\sigma \int_{B_R(y)} \left( 1+ |Du|^{p_2(1+\vep)}
	+\left( 1+\tilde{a}(x)|Du|^{q_2}\right)^{1+\vep}\right) dx\nn\\
	\leq ~& CR^{n+\sigma}+ C(\vep)R^\sigma
	\int_{B_R(y)} F_0(x,Du)^{1+\vep} dx\la{est35}
%	\leq ~&CR^{n+\sigma}+ C(\vep)R^\sigma
%	\int_{B_R(y)} F(x,Du)^{1+2\vep} dx.
\end{align}
Similarly we have
\begin{align}
	& \F(v,B_R(y))-\F_0(v)\nn\\
		\leq & \int_{B_R(y)} \left[\left( |D v|^{p_2} -|Dv|^{p(x)}\right) +
	\left( \left(a(x)^{\frac{1}{q(x)}}|Dv|\right)^{q_2}- \left(a(x)^{\frac{1}{q(x)}} |Dv|\right)^{q(x)}\right)\right] dx \nn\\
	\leq & C(\vep)R^\sigma \int_{B_R(y)} \left( 1+|Dv|^{(1+\vep)p_2}\right) dx\nn\\
	& ~~ ~~ ~~ ~~ + C(\vep) R^\sigma \int_{B_R(y)}
	\left( 1+\left(a(x)^{\frac{1}{q(x)}}|Dv|\right)^{(1+\vep)q_2}\right) dx\nn\\
	\leq ~ &CR^{n+\sigma}+ C(\vep)R^\sigma \int_{B_R(y)} \left( 1+ |Dv|^{p_2(1+\vep)}
	+\left( 1+\tilde{a}(x)|Dv|^{q_2}\right)^{1+\vep}\right) dx\nn\\
	\leq ~& CR^{n+\sigma}+ C(\vep)R^\sigma
	\int_{B_R(y)} F_0(x,Dv)^{1+\vep} dx.\la{est36}
\end{align}

Now, for  $\delta_0$ of Proposition \ref{higher-int}, choose $\delta_3>0$ so that
\pref{2.28} of Corollary \ref{higher-global} holds,  and let us take
$\vep$ so that $\vep \in (0, \min\{\delta_0/2, \delta_3\}/2)$.
Since  we are choosing $R$ so that \pref{choice of R} holds, we have
\be\la{3.10}
	F_0(x,\cdot)^{1+\vep}\leq (1+ F_0(x,\cdot))^{1+\min\{\delta_0/2, \delta_3\}}\leq C(1+F(x, \cdot))^{1+\delta_0}.
\ee

By Proposition \ref{higher-int} and \pref{3.10}, we deduce from \pref{est35} that
\begin{align}
	& \F_0(u) -\F(u,B_R(y))\nn\\
	\leq~&
	CR^{n+\sigma}+ C(\vep)R^\sigma \int_{B_R(y)} \left(1+F(x,Du)\right)^{1+\delta_0} dx \nn\\
	\leq ~& CR^{n+\sigma}+ C R^\sigma \int_{B_R(y)} F(x,Du)^{1+\delta_0} dx \nn\\
	\leq ~& CR^{n+\sigma}+ C  R^{\sigma-n\vep}
	\left( \int_{B_{2R}(y)} F(x,Du)dx \right)^{1+  \delta_0} \nn\\
	\leq ~& CR^{n+\sigma} + C R^{\sigma-n\vep}
	\int_{B_{2R}(y)} F(x,Du)dx,
	\la{est38}
\end{align}
where we used the fact that
\[	
	\int_{B_{2R}(y)} F(x,Du)dx \leq \int_K F(x,Du)dx \leq M_0
\]
for some constant $M_0$. The existence of $M_0$  guaranteed by the local minimality of
$u$.

For \pref{est36} we use Proposition \ref{higher-global}, Proposition \ref{higher-int} and \pref{3.10}, to get
\begin{align}
	& \F(v,B_R(y))-\F_0(v)\nn\\
	\leq ~& CR^{n+\sigma}+ C(\vep)R^\sigma\int_{B_R(y)} F_0(x,Du)^{1+\vep} dx\nn\\
	\leq ~& CR^{n+\sigma} + C R^{\sigma-n\vep}
	\int_{B_{2R}(y)} F(x,Du)dx.\la{est39}
\end{align}
On the other hand, by the definition of $F_0$, we have
\[
	F(x,Du) \leq C\left(1+F_0(x,Du)\right).
\]
So we have, combining \pref{est33}, \pref{est34}, \pref{est38} and \pref{est39}, that
\begin{align}
	 &\int_{B_R(y)}\left(|Du|^{p_2-2}+ |Dv|^{p_2-2}\right)
	|Du-Dv|^2 dx\nn\\
	& ~~ ~~ ~~ ~~
	+ \int_{B_R(y)}\tilde{a}(x) \left(|Du|^{q_2-2}+ |Dv|^{q_2-2}\right)
	|Du-Dv|^2 dx\nn\\
	\leq ~& \F_0(u)-\F_0(v) \nn\\
	\leq ~&CR^{n+\sigma} + CR^{\sigma-n \vep} \int_{B_{2R}(y)} (1+F_0(x,Du)) dx.
	\la{est42}
\end{align}

%Now, let us consider the case that $p_2\geq 2$
By virtue of \pref{V_p} and \pref{morrey-1}, we can see that
\begin{align}
	& \int_{B_\rho(y)} (1+ F_0(x,Du)) dx \nn\\
	=~& \int_{B_\rho(y)} (1+ F_0(x,Dv) ) dx+\int_{B_\rho(y)}\left(F_0(x,Du)-F_0(x,Dv)\right)dx\nn\\
	\leq ~ & C\left(\frac{\rho}{R}\right)^{n-\gamma}
	\int_{B_(y)} ( 1+ F_0(x,Dv))  dx \nn\\
%	& +C \int_{B_R(y)} |Du-Dv|^{p_2} dx +
%	C \int_{B_R(y)} \tilde{a}(x) |Du-Dv|^{q_2} dx\nn\\
	& + \int_{B_\rho(y)} \left[ |V_{p_2} (Du)|^2 + \tilde{a}(x) |V_{q_2}(Du)|^2
	-\left(|V_{p_2} (Dv)|^2 +\tilde{a}(x)|V_{q_2}(Dv)|^2\right)\right] dx\nn\\
	\leq ~& C\left(\frac{\rho}{R}\right)^{n-\gamma}
	\int_{B_(y)} (1+ F_0(x,Dv))  dx \nn\\
	& + \int_{B_R(y)}\left[ \left( |V_{p_2}(Du) |^2 - |V_{p_2}(Dv)|^2\right)
	+ \tilde{a}(x) \left( |V_{q_2}(Du) |^2 - |V_{q_2}(Dv)|^2\right)\right] dx\nn\\
	\leq ~& C\left(\frac{\rho}{R}\right)^{n-\gamma}
	\int_{B_(y)} ( 1+ F_0(x,Dv)) dx \nn\\
	& + \int_{B_R(y)} |V_{p_2}(Du) - V_{p_2}(Dv)|^2dx
	+ \int_{B_R(y)}  \tilde{a}(x) |V_{q_2}(Du) -V_{q_2}(Dv)|^2 dx\nn\\
	\leq ~ & C\left(\frac{\rho}{R}\right)^{n-\gamma}
	\int_{B_(y)} ( 1+ F_0(x,Dv)) dx \nn\\
	& +\int_{B_R(y)}\left(|Du|^{p_2-2}+ |Dv|^{p_2-2}\right)
	|Du-Dv|^2 dx\nn\\
	& ~~ ~~ ~~ ~~
	+ \int_{B_R(y)}\tilde{a}(x) \left(|Du|^{q_2-2}+ |Dv|^{q_2-2}\right)
	|Du-Dv|^2 dx\nn\\
	\leq ~& C\left(\frac{\rho}{R}\right)^{n-\gamma}
	\int_{B_{R}(y)} (1 + F_0(x,Dv))  dx \nn\\
	& ~~~ + CR^{n+\sigma} +
	CR^{\sigma- n \vep} \int_{B_{2R}(y)} (1+F_0(x,Du)) dx\nn\\
	\leq ~& C\left[\left(\frac{\rho}{R}\right)^{n-\gamma} + R^{\sigma-n \vep}\right]
	\int_{B_{2R}(y)} (1+ F_0(x,Du))  dx+CR^{n+\sigma}.
\end{align}
Using well-known lemma (see for example \cite[Lemma 5.13]{giamar05}), for sufficiently small $R>0$,
we can see that for any %\color{cyan}(\fbox{AT} Nostro sbaglio $\Rightarrow$ $\zeta\in (0,\gamma)$ deve
%eseere $\zeta \in (\gamma, 1)$. Per quest vorrei usare un altro carattere. Cioe vorrei scrivere come
 $ \gamma^\prime \in (\gamma,1)$  there exists a constant $C$ depending
given data and $\zeta$ such that
\be\la{est44}
	\int_{B_\rho(y)}F_0(x,Du)dx\leq C\left(\frac{\rho}{R}\right)^{n- \cbl\gamma^\prime}
	\int_{B_{2R}(y)} F_0(x,Du) dx+C\rho^{n-\gamma^\prime}
\ee
hold for any $\rho \in (0,R)$.
Now, since \pref{morrey-1} holds for any $\gamma \in (0,1)$, we can choose
$\gamma^\prime \in (0,1)$ arbitrarily in \pref{est44}.
%\color{cyan} \fbox{AT}
%(Qui mettiamo un po di spiegazione per 1-st reviewer.)\\
%\cm
On the other hand, since we are supposing that $p(x)\geq p_0>1$,
for any $\zeta \in (0,1)$, choosing $\gamma^\prime \in (0,1)$
so that $\gamma^\prime \leq  p_0(1-\zeta)$,  we see that there exists
a positive constant $C$ dependent on the given data,
$K \Subset \Omega$ and $\F(u,K)$ such that
\[
	\int_{B_\rho(y)} |Du|^{p_0} dx \leq C\rho^{n-p_0(1-\zeta)}
\]
holds for any $B_\rho(y)$ with $4\rho \leq \dist (K,\partial \Omega)$.
%\cb
So, we conclude that
$u \in C^{0,\zeta}_{\loc}(\Omega)$
for any $\zeta \in (0,1)$ by virtue of Morrey's theorem.
%\end{proof}

%\begin{thm}\la{C1gamma}
%	 $u \in C^{1,\gamma}(\Omega)$ for some $\gamma \in (0,1)$.
%\end{thm}
%\begin{proof}

\medskip
\noindent
{\bf Part 2.}
Now, we are going to show the H\"older continuity of the gradient $Du$.
%%%%%%%
%%%%%%
%\cred%\fbox{AT}
%(Usiamo $R_1$ invece $R_0$. Perche` abbiamo fissato $R_0$ come Mingione ha suggerito.)
For $y\in \stackrel {\circ }{K}$ let $R_1\in (0,R_0)$ be a constant such that $B_{R_1}(y) \subset K$,
and for $0<R<R_1/4$ let $v$ be as in {\it Part 1}.
\cb
%Let $1>R_0>0$ be a constant  such that $B_{R_0}(y) \subset K$,
%and for $0<R<R_0/4$ let $v$ be as in {\it Part 1}.
Then, by the estimate given by Colombo-Mingione at \cite[p.484, l.-6]{colmin15-1},
we see that there exist constants $C>0$
\!\!, dependent on $n, p_2, q_2, \lambda, \Lambda,  \| \tilde{a}\|_\infty$,
%(\fbox{AT} depende anche $\Rightarrow$)
$\dist(K,\partial \Omega)$, $\F_0(v,B_R(y))$
 and $\tilde{\alpha}\in (0,1)$
\be\la{colminp484}
	\intmean_{B_\rho(y)} |Dv-(Dv)_\rho|^{p_2}dx \leq C\rho^{\frac{\tilde{\alpha}\beta}{64n}},
\ee
holds for any  $\rho \leq R/2$.
%%%%%%%
%%%%%%%
%\cred%\fbox{AT} (Mettiamo un po di spiegazione anche qui.)
Here, as in \bf Part 1\rm, let us mention that $\F_0(v,B_R(y))$ can be controlled by $\F(u,K)$
as \pref{2019-3}.
So, we can choose the above constant in \pref{colminp484}
to be dependent only on the given data of the functional,
the local minimizer $u$ under consideration  and $K$.
%\cb

In what follows, let us abbreviate
\[
	\bar{\alpha}:=\frac{\tilde{\alpha}\beta}{64n}.
\]

By virtue of \pref{colminp484}, for $\rho$ and $R$ as above,
we get
\begin{align}
	& \int_{B_\rho(y)}|Du-(Du)_\rho|^{p_2} dx
	\leq C\int_{B_\rho(y)} |Du-(Dv)_\rho|^{p_2} dx \nn\\
	\leq ~& C\left( \int_{B_\rho(y)} \left|Dv-(Dv)_\rho\right|^{p_2} dx
	+ C\int_{B_\rho(y)} \left|Du-Dv\right|^{p_2} dx \right)\nn\\
	\leq ~ & C \rho^{n+\bar{\alpha}}
	+ C\int_{B_R (y)} \left|Du-Dv\right|^{p_2} dx.
	\la{est46}
\end{align}

For the case that $p_2 \geq 2$, since there exists a constant such that
\[
	|z_1-z_2|^{p_2}\leq C \left(|z_1|^{p_2-2} + |z_2|^{p_2-2} \right)|z_1-z_2|^2
\]
for any $z_1,  z_2\in \R^n$, using \pref{est42},
we can estimate the last term of the right hand side of \pref{est46}
as
\begin{align}
	& \int_{B_R(y)}|Du-Dv|^{p_2} dx\nn\\
	\leq ~& CR^{n+\sigma} +C R^{\sigma-n \vep}
	\int_{B_{2R}(y)} F_0(x, Du)dx.\la{est46-1}
\end{align}
We use \pref{est44} replacing $\rho$ by $2R$ and $R$ by $R_0$ to see that
\[
	\int_{B_{2R}(y)}F_0(x,Du)dx\leq C R^{n-\zeta} R_0^\zeta \intmean_{B_{R_0}}
	F_0(x,Du)dx + CR^{n-\zeta}.
\]
Since $R_0$ is determined in the beginning of the proof, we can regard
$R_0^\zeta ~ \sintmean_{B_{R_0}} F_0(x,Du)dx$ as a constant.
So, we get
\be\la{est46-3}
	\int_{B_{2R}(y)}F_0(x,Du)dx\leq CR^{n-\zeta}.
\ee
By \pref{est46-1} and \pref{est46-3}, we
obtain
\be\la{est1-p2geq2}
	\int_{B_R(y)}|Du-Dv|^{p_2} dx
	\leq CR^{n+\sigma} +C R^{n-\zeta+\sigma-n \vep} \leq CR^{n-\zeta+\sigma-n  \vep}.
\ee

When $1<p_2<2$, using H\"older's inequality, \pref{V_p}
and \pref{est42},
we can see that
\begin{align}
	& \int_{B_R(y)} |Du-Dv|^{p_2} dx\nn\\
	\leq ~& C \int_{B_R(y)} \left|V_{p_2}(Du)-V_{p_2}(Dv)\right|^{p_2}
	(|Du|+|Dv|)^{\frac{p_2(2-p_2)}{2}} dx\nn\\
	\leq ~& C \left(\int_{B_R(y)} \left|V_{p_2}(Du)-V_{p_2}(Dv)\right|^2 dx\right)^{\frac{p_2}{2}}
	\left(\int_{B_R(y)} (|Du|+|Dv|)^{\frac{p_2}{2}} dx\right)^{\frac{2-p_2}{2}}\nn\\
	\leq ~& \left( \int_{B_R(y)} (|Du|+|Dv|)^{p_2-2} |Du-Dv|^2 dx \right)^{p_2}
	\left(\int_{B_R(y)} F_0(x,Du) dx\right)^{\frac{2-p_2}{2}}\nn\\
	\leq ~& \left( CR^{n+\sigma} +C R^{\sigma -n \vep}
	\int_{B_{2R}(y)} F_0 (x,Du)dx \right)^{\frac{p_2}{2}}
	\left(\int_{B_{2R}(y)} F_0(x,Du) dx\right)^{\frac{2-p_2}{2}}\nn\\
	\leq ~& CR^{\frac{(n+\sigma)p_2}{2}}\left(\int_{B_{2R}(y)} F_0(x,Du) dx\right)^{\frac{2-p_2}{2}}
	\nn\\
	& ~~ ~~ ~~ + CR^{\frac{(\sigma -n\vep)p_2}{2}} \int_{B_{2R}(y)}F_0(x,Du)dx.
	\la{est46-2}
\end{align}
By \pref{est46-2} and \pref{est46-3},
we obtain
\begin{align}
	 & \int_{B_R(y)} |Du-Dv|^{p_2} dx\nn\\
	\leq ~& CR^{\frac{p_2(n+\sigma)}{2}} R^{\frac{(2-p_2)(n-\zeta)}{2}} +
	CR^{\frac{(\sigma -n  \vep )p_2}{2}} R^{n-\zeta}\nn\\
	= ~& CR^{n-\zeta+\frac{p_2(\sigma+\zeta)}{2}} +
	CR^{n-\zeta+\frac{p_2(\sigma-n  \vep )}{2}}\nn\\
	\leq ~&  2 CR^{n-\zeta+\frac{p_2(\sigma-n  \vep)}{2}}
	\leq 2 CR^{n-\zeta+\frac{(\sigma-n \vep)}{2}}.
	\la{est1-p2<2}
\end{align}
For the last inequality we used the following facts:
\[
	0<R \leq 1, ~~0<\sigma - n \vep, ~~ p_2>1.
\]

Mentioning the above facts again and comparing \pref{est1-p2geq2} and \pref{est1-p2<2},
we see that, for $p_2>2$, the estimate  \pref{est1-p2<2} holds.
Now, combining \pref{est46} and \pref{est1-p2<2}, we obtain
\begin{align*}
	& \int_{B_\rho(y)}|Du-(Du)_\rho|^{p_2} dx
%	\leq ~& C\left( \rho^{n+\bar{\alpha}} + R^{n+\sigma}  + R^{n-\zeta +(\sigma - n\cred  \vep \cb)/2}\right)
	\leq  C\left( \rho^{n+\bar{\alpha}} + R^{n-\zeta+\frac{\sigma - n  \vep}{2}}\right).%\la{est49}
\end{align*}
This holds for any $0< \rho < R/2 \leq R_0/8$.
For $k>1$, let us put $\rho = R^k/2$
(%remark
%\cm
%invece di "remark" scriviamo "
bearing in mind %" ?
%\cb
that $R^k/2\leq R/2$  holds for $k>1$), then
\[
	\rho^{n+\bar{\alpha}} + R^{n-\zeta +\frac{\sigma - n \vep }{2}} =
	\rho^{n+\bar{\alpha}} + (2\rho)^{\frac{2n-2\zeta +\sigma - n  \vep}{2k}}.
\]
So, we have
\begin{align}
	& \int_{B_\rho(y)}|Du-(Du)_\rho|^{p_2} dx
	\leq \rho^{n+\bar{\alpha}} + (2\rho)^{\frac{2n-2\zeta +\sigma - n  \vep}{2k}}.
	\la{est54}
\end{align}
Since
\[
	\bar{\alpha} =\frac{\tilde{\alpha}}{64n} \beta = \frac{\tilde{\alpha}}{64n} \min\{\alpha, \sigma\}
	\leq \frac{\sigma}{64},
\]
we can take $\vep $ sufficiently small so that $\bar{\alpha}< (\sigma-n \vep)/2$ then,
for sufficiently small $\zeta$,
\[
	n-\zeta+\frac{\sigma-n\vep}{2} > n+\bar{\alpha}
\]
holds.
Now, for such a choice of $ \vep $ and $\zeta$, putting
\[
	k=\frac{2n-2\zeta+\sigma-n  \vep}{2(n+\bar{\alpha})}~(>1)
\]
in \pref{est54},
we get
\[	
	\int_{B_\rho(y)}|Du-(Du)_\rho|^{p_2} dx\leq C \rho^{n+\bar{\alpha}},
\]
and therefore we obtain the H\"older continuity of $Du$ by virtue of
the Campanato's theorem.
\qed

{\bf Acknowledgement} The authors are deeply grateful to Giuseppe Mingione for interesting them in the problem.
This paper was partly prepared while the authors visited in Pisa the \it Centro di Ricerca Matematica Ennio De Giorgi \rm
- \rm Scuola Normale Superiore \rm in September 2016. The hospitality of the center is greatly acknowledged.

%\bibliographystyle{abbrv}
%\bibliography{miabib-1}

\begin{thebibliography}{10}

%[1] T. Angel and I. E. Shparlinski, An article in lower case letters and upright, J. Math. Crypt. 23, (2007), no. 1, 1–11.
%
%[2] L.A. Kurdachenko, N.N. Semko, I.Ya. Subbotin, The Leibniz algebras whose subalgebras are ideals, Open Math. 15, (2017), 92–100.
%
%[3] S. Schmitt and H. G. Zimmer, A Book in Upper Case Letters and Italics, 2nd ed., De Gruyter Stud. Math. 31, De Gruyter, Berlin, 2010.
%
%[4] N. Gigli (Ed.), Measure Theory in Non-Smooth Spaces, De Gruyter Publishing Group, 2017.

\bibitem{colmin15-1} M.~Colombo and G.~Mingione,
\newblock Regularity for double phase variational problems.
\newblock {\em Arch. Ration. Mech. Anal.}, {\bf 215} (2), (2015), 443--496{\noop{2014}}.

\bibitem{zhi86(rus)} V.~V. Zhikov,
\newblock Averaging of functionals of the calculus of variations and elasticity theory.
\newblock {\em Izv. Akad. Nauk SSSR Ser. Mat.}, {\bf 50} (4), (1986), 675--710. %, 877.

\bibitem{zhi95} V.~V. Zhikov,
\newblock On {L}avrentiev's phenomenon.
\newblock {\em Russian J. Math. Phys.}, {\bf 3} (2), (1995), 249--269.

\bibitem{zhi97} V.~V. Zhikov,
\newblock On some variational problems.
\newblock {\em Russian J. Math. Phys.}, {\bf 5} (1997), 105--116.

\bibitem{ruz00(LN)} M.~R{\ocirc{u}}{\v{z}}i{\v{c}}ka,
\newblock {\em Electrorheological fluids: modeling and mathematical theory}, {\bf 1748} of {\em Lecture Notes in Mathematics}.
\newblock Springer-Verlag, Berlin, 2000.

\bibitem{acemin02} E.~Acerbi and G.~Mingione,
\newblock Regularity results for stationary electro-rheological fluids.
\newblock {\em Arch. Ration. Mech. Anal.}, {\bf 164} (3), (2002), 213--259.

\bibitem{rajruz01} K.~Rajagopal and M.~R{\ocirc{u}}{\v{z}}i{\v{c}}ka,
\newblock Mathematical modeling of electrorheological materials.
\newblock {\em Contin. Mech. Thermodyn.}, {\bf 13}, (2001), 59--78.

\bibitem{bogduzhabsch12} V.~B{\"o}gelein, F.~Duzaar, J.~Habermann and C.~Scheven,
\newblock Stationary electro-rheological fluids: low order regularity for systems with discontinuous coefficients.
\newblock {\em Adv. Calc. Var.}, {\bf 5} (1), (2012), 1--57.

\bibitem{mar89} P.~Marcellini,
\newblock Regularity of minimizers of integrals of the calculus of variations with nonstandard growth conditions.
\newblock {\em Arch. Rational Mech. Anal.}, {\bf 105} (3), (1989), 267--284.

\bibitem{mar91} P.~Marcellini,
\newblock Regularity and existence of solutions of elliptic equations with {$p,q$}-growth conditions.
\newblock {\em J. Differential Equations}, {\bf 90} (1), (1991), 1--30.

\bibitem{mar96(SNS)} P.~Marcellini,
\newblock Everywhere regularity for a class of elliptic systems without growth conditions.
\newblock {\em Ann. Scuola Norm. Sup. Pisa Cl. Sci. (4)}, {\bf 23} (1), (1996), 1--25.

\bibitem{BIFUCHS} M.~Bildhauer and M.~Fuchs,
\newblock {$C^{1,\alpha}$}-solutions to non-autonomous anisotropic variational  problems.
\newblock {\em Calc. Var. Partial Differential Equations}, {\bf 24} (3), (2005), 309--340.

\bibitem{BREIT} D.~Breit,
\newblock New regularity theorems for non-autonomous variational integrals with
  {$(p,q)$}-growth.
\newblock {\em Calc. Var. Partial Differential Equations}, {\bf 44} (1-2), (2012), 101--129.

\bibitem{CHOE} H.~J. Choe,
\newblock Interior behaviour of minimizers for certain functionals with
  nonstandard growth.
\newblock {\em Nonlinear Anal.}, {\bf 19} (10), (1992),933--945.

\bibitem{espleomin04} L.~Esposito, F.~Leonetti and G.~Mingione,
\newblock Sharp regularity for functionals with {$(p,q)$} growth.
\newblock {\em J. Differential Equations}, {\bf 204} (1), (2004), 5--55.

\bibitem{sch08(advcal)} T.~Schmidt,
\newblock Regularity theorems for degenerate quasiconvex energies with {$(p,q)$}-growth.
\newblock {\em Adv. Calc. Var.}, {\bf 1} (3), (2008), 241--270.

\bibitem{sch08(calvar)} T.~Schmidt,
\newblock Regularity of minimizers of {$W\sp {1,p}$}-quasiconvex variational
  integrals with {$(p,q)$}-growth.
\newblock {\em Calc. Var. Partial Differential Equations}, {\bf 32} (1), (2008),1--24.

\bibitem{sch09} T.~Schmidt,
\newblock Regularity of relaxed minimizers of quasiconvex variational integrals with {$(p,q)$}-growth.
\newblock {\em Arch. Ration. Mech. Anal.}, {\bf 193} (2), (2009),311--337.

\bibitem{URUR} N.~N. Ural\cprime~tseva and A.~B. Urdaletova,
\newblock Boundedness of gradients of generalized solutions of degenerate nonuniformly elliptic quasilinear equations.
\newblock {\em Vestnik Leningrad. Univ. Mat. Mekh. Astronom.}, (vyp. {\bf 4}), (1983), 50--56.

\bibitem{min06} G.~Mingione,
\newblock Regularity of minima: an invitation to the dark side of the calculus of variations.
\newblock {\em Appl. Math.}, {\bf 51} (4), (2006),355--426.

\bibitem{cosmin99} A.~Coscia and G.~Mingione,
\newblock H\"older continuity of the gradient of {$p(x)$}-harmonic mappings.
\newblock {\em C. R. Acad. Sci. Paris S\'er. I Math.}, {\bf 328} (4), (1999), 363--368.

\bibitem{acemin01} E.~Acerbi and G.~Mingione,
\newblock Regularity results for a class of quasiconvex functionals with nonstandard growth.
\newblock {\em Ann. Scuola Norm. Sup. Pisa Cl. Sci. (4)}, {\bf 30} (2), (2001), 311--339.

\bibitem{acemin01(arc)} E.~Acerbi and G.~Mingione,
\newblock Regularity results for a class of functionals with non-standard growth.
\newblock {\em Arch. Ration. Mech. Anal.}, {\bf 156} (2), (2001), 121--140.

\bibitem{ele04} M.~Eleuteri,
\newblock H\"older continuity results for a class of functionals with  non-standard growth.
\newblock {\em Boll. Unione Mat. Ital. Sez. B Artic. Ric. Mat. (8)}, {\bf 7} (1), (2004), 129--157.

\bibitem{CRR} M.~Cencelj, V.~R\u{a}dulescu and D.~Repov\u{s},
\newblock Double phase problems with variable growth.
\newblock {\em Nonlinear Analysis}, {\bf 177}, (2018), 270--287.

\bibitem{EMM1} M.~Eleuteri, P.~Marcellini and E.~Mascolo,
\newblock Lipschitz estimates for systems with ellipticity conditions at infinity.
\newblock {\em Ann. Mat. Pura Appl.}, {\bf 195}, (2016), 1575-1603.

\bibitem{EMM2} M.~Eleuteri, P.~Marcellini and E.~Mascolo,
\newblock Regularity for scalar integrals without structure conditions
\newblock {\em Adv. Calc. Var.}, in press. %, 2018.

\bibitem{EMM3} M.~Eleuteri, P.~Marcellini and E.~Mascolo,
\newblock Lipschitz continuity for energy integrals with variable exponents
\newblock {\em Rend. Lincei - Matematica e Appl.}, {\bf 27}, (2016), 61--87.

\bibitem{min07} G.~Mingione,
\newblock Short tales from nonlinear calder\'on-{Z}ygmund theory.
\newblock {\em Lecture Notes in Mathematics.}, {\bf 2186}, (2017), 159--204.

\bibitem{PRR} N.S.~Papageorgiou, V.~R\u{a}dulescu and D.~Repov\u{s},
\newblock Double-phase problems with reaction of arbitrary growth
\newblock {\em Zeitschrift fuer angewandte Mathematik und Physik (ZAMP)}, {\bf 69}, (2018), 108.

\bibitem{PR1}  V.~R\u{a}dulescu and D.~Repov\u{s},
\newblock Partial Differential Equations with Variable Exponents: Variational Methods and Qualitative Analysis
\newblock {\em CRC Press, Taylor Francis Group, Boca Raton FL}, 2015.

\bibitem{PR2}  V.~R\u{a}dulescu and  Q.~Zhang,
\newblock Double phase anisotropic variational problems and combined effects of reaction and absorption terms
\newblock {\em  J. Math. Pures Appl.}, {\bf 118}, (2018), 159--203.

\bibitem{PR3} V.~R\u{a}dulescu,
 \newblock Nonlinear elliptic equations with variable exponent: old and new
\newblock {\em Nonlinear Analysis: Theory, Methods and Applications}, {\bf 121}, (2015), 336--369.

\bibitem{PR4} V.~R\u{a}dulescuu,
\newblock  Isotropic and anisotropic double-phase problems: old and new
\newblock {\em Opuscula Mathematica}, {\bf 39}, (2019), 259--279.

\bibitem{PR5} V.~R\u{a}dulescu, D.~Repov\u{s}, X.~Shi and Q.~Zhang,
\newblock  Multiple solutions of double phase variational problems with variable exponent
\newblock {\em Advances in Calculus of Variations}, doi.org/10.1515/acv-2018-0003, in press.

\bibitem{ragtactak} M.~A. Ragusa, A.~Tachikawa and H.~Takabayashi,
\newblock Partial regularity of {$p(x)$}-harmonic maps.
\newblock {\em Trans. Amer. Math. Soc.}, {\bf 365} (6), (2013), 3329--3353.

\bibitem{ragtac13} M.~A. Ragusa and A.~Tachikawa,
\newblock On interior regularity of minimizers of {$p(x)$}-energy functionals.
\newblock {\em Nonlinear Anal.}, {\bf 93}, (2013), 162--167.

\bibitem{tac14} A.~Tachikawa,
\newblock On the singular set of minimizers of {$p(x)$}-energies.
\newblock {\em Calc. Var. Partial Differential Equations}, {\bf 50} (1-2), (2014), 145--169.

\bibitem{usu15} K.~Usuba,
\newblock Partial regularity of minimizers of {$p(x)$}-growth functionals with {$1<p(x)<2$}.
\newblock {\em Bull. Lond. Math. Soc.}, {\bf 47} (3), (2015), 455--472.

\bibitem{ragtac16} M.~A. Ragusa and A.~Tachikawa,
\newblock Boundary regularity of minimizers of {$p(x)$}-energy functionals.
\newblock {\em Ann. Inst. H. Poincar\'e Anal. Non Lin\'eaire}, {\bf 33} (2), (2016), 451--476.

\bibitem{tacusu16} A.~Tachikawa and K.~Usuba,
\newblock Regularity results up to the boundary for minimizers of p(x)-energy with {$p(x)>1$}.
\newblock {\em Manuscripta Math.}, {\bf 152} (1-2), (2017), 127--151.

\bibitem{colmin15-2} M.~Colombo and G.~Mingione,
\newblock Bounded minimisers of double phase variational integrals.
\newblock {\em Arch. Ration. Mech. Anal.}, {\bf 218} (1), ( 2015), 219--273.

\bibitem{colmin16} M.~Colombo and G.~Mingione,
\newblock Calder\'on-{Z}ygmund estimates and non-uniformly elliptic operators.
\newblock {\em J. Funct. Anal.}, {\bf 270} (4), (2016), 1416--1478.

\bibitem{barcolmin16(Petersburg)} P.~Baroni, M.~Colombo and G.~Mingione,
\newblock Nonautonomous functionals, borderline cases and related function classes.
\newblock {\em St.Petersburg Math. J.}, {\bf 27} (3), ( 2016), 347--379.

\bibitem{barcolmin16} P.~Baroni, M.~Colombo, and G.~Mingione,
\newblock Harnack inequalities for double phase functionals.
\newblock {\em Nonlinear Anal.}, {\bf 121}, (2015), 206--222.

\bibitem{barcolmin18} P.~Baroni, M.~Colombo, and G.~Mingione,
\newblock Regularity for general functionals with double phase.
\newblock {\em Calc. Var. Partial Differential Equations}, {\bf 57} (2), (2018), art. 62, 48.

\bibitem{HHT}  P.~Harjulehto, P.~H\"{a}st\"{o} and O.~Toivanen,
\newblock H\"older regularity of quasiminimizers under generalized growth conditions,
\newblock {\em Calc. Var. Partial Differential Equations}, {\bf 56}, (2017),  no. 2, article 22.

\bibitem{cosmuc02} A.~Coscia and D.~Mucci,
\newblock Integral representation and {$\Gamma$}-convergence of variational integrals with {$P(X)$}-growth.
\newblock {\em ESAIM Control Optim. Calc. Var.}, {\bf 7}, (2002), 495--519 (electronic).

\bibitem{dac15(book)} B.~Dacorogna,
\newblock {\em Introduction to the calculus of variations}.
\newblock Imperial College Press, London, (2015), third edition.

\bibitem{eva10(book)} L.~C. Evans,
\newblock {\em Partial differential equations}, volume~19 of {\em Graduate Studies in Mathematics}.
\newblock American Mathematical Society, Providence, RI, (2010), second edition.

\bibitem{giltru98} D.~Gilbarg and N.~S. Trudinger,
\newblock {\em Elliptic partial differential equations of second order, \rm (2nd edition, revised 3rd printing)}.
\newblock Springer Verlag, 1998.

\bibitem{malzie97} J.~Mal\'y and W.~P. Ziemer,
\newblock  Fine regularity of solutions of elliptic partial differential equations, volume~51 of Mathematical Surveys and Monographs.
\newblock {\em  American Mathematical Society, Providence, RI}, 1997.

\bibitem{dieett08} L.~Diening and F.~Ettwein,
\newblock Fractional estimates for non-differentiable elliptic systems with general growth.
\newblock {\em Forum Math.}, {\bf 20} (3), (2008), 523--556.

\bibitem{HHK} P.~Harjulehto, P.~H\"{a}st\"{o} and A.~Karppinen,
\newblock  Local higher integrability of the gradient of a quasiminimizer under generalized Orlicz growth conditions.
\newblock {\em   Nonlinear Anal.}, {\bf 177}, (2018), 543--552.

\bibitem{giu03} E.~Giusti,
\newblock {\em Direct methods in the calculus of variations}.
\newblock World Scientific Publishing Co. Inc., River Edge, NJ, 2003.

\bibitem{DHHR} L.~Diening, P.~Harjulehto, P.~H\"ast\"o and M.~R{\ocirc{u}}{\v{z}}i{\v{c}}ka,
\newblock Lebesgue and Sobolev spaces with variable exponents.
\newblock {\em Lecture Notes in Mathematics}, {\bf 2017}, (2011),  Springer-Verlag, Heidelberg.

\bibitem{giamod79} M.~Giaquinta and G.~Modica,
\newblock Regularity results for some classes of higher order nonlinear elliptic systems.
\newblock {\em J. Reine Angew. Math.}, {\bf 311/312}, (1979), 145--169.

\bibitem{DO} C.~De Filippis and J.~Oh,
\newblock Regularity for multi-phase variational problems.
\newblock  {\em Submitted. https://arxiv.org/abs/1807.02880.}

\bibitem{DM}  C.~De Filippis and G.~Mingione,
\newblock  Manifold constrained non-uniformly elliptic problems.
\newblock  {\em Preprint}.

\bibitem{giamar05} M.~Giaquinta and L.~Martinazzi,
\newblock {\em An introduction to the regularity theory for elliptic systems, harmonic maps and minimal graphs}, volume~2 of {\em Appunti. Scuola Normale  Superiore di Pisa (Nuova Serie) [Lecture Notes. Scuola Normale Superiore di  Pisa (New Series)]}.
\newblock Edizioni della Normale, Pisa, 2005.

\end{thebibliography}
\newcommand{\noop}[1]{}\def\ocirc#1{\ifmmode\setbox0=\hbox{$#1$}\dimen0=\ht0
  \advance\dimen0 by1pt\rlap{\hbox to\wd0{\hss\raise\dimen0
  \hbox{\hskip.2em$\scriptscriptstyle\circ$}\hss}}#1\else {\accent"17 #1}\fi}
  \def\cprime{$'$} \def\cprime{$'$}

\end{document}